\documentclass[11pt,a4paper,twoside]{amsart}

\usepackage{palatino}
\usepackage[all]{xy,xypic}
\usepackage{amsfonts,amsmath,oldgerm,amssymb,amscd,tikz-cd,mathrsfs,enumerate,amsthm,bm,longtable}
\usepackage{geometry}
\usepackage{hyperref}
\usepackage[english]{babel}
\usepackage{paralist}

\usepackage{fancyhdr}
\newcommand{\remove}[1]{}

\newcommand{\inj}{\hookrightarrow}
\newcommand{\ol}{\overline}

\newtheorem{theorem}{Theorem}[section]
\newtheorem{proposition}[theorem]{Proposition}
\newtheorem{lemma}[theorem]{Lemma}
\newtheorem{definition}[theorem]{Definition}
\newtheorem{corollary}[theorem]{Corollary}
\newtheorem{conjecture}[theorem]{Conjecture}
\newtheorem{example}[theorem]{Example}

\newcommand{\mm}{\mbox{$\mathfrak m$}}	
	
\newcommand{\mq}{\mbox{$\mathfrak q$}}

\newcommand{\Spec}{\text{Spec}}

\newcommand{\A}{\text{Alt}}

\newcommand{\bp}{\begin{proposition}}
	\newcommand{\ep}{\end{proposition}}
\newcommand{\bl}{\begin{lemma}}
	\newcommand{\el}{\end{lemma}}
\newcommand{\bt}{\begin{theorem}}
	\newcommand{\et}{\end{theorem}}
\newcommand{\bc}{\begin{corollary}}
	\newcommand{\ec}{\end{corollary}}
\newcommand{\bd}{\begin{definition}}
	\newcommand{\ed}{\end{definition}}
\newcommand{\bco}{\begin{conjecture}}
	\newcommand{\eco}{\end{conjecture}}
\newcommand{\bma}{\begin{bmatrix}}
	\newcommand{\ema}{\end{bmatrix}}

\newcommand{\Um}{\mbox{\rm Um}}		\newcommand{\SL}{\mbox{\rm SL}}
 \newcommand{\GL}{\mbox{\rm {GL}}}	

\newcommand{\Sp}{\mbox{\rm Sp}}
\newcommand{\ESp}{\mbox{\rm ESp}}
\newcommand{\M}{\mbox{\rm M}}
\newcommand{\I}{\mbox{\rm I}}
\newcommand{\E}{\mbox{\rm E}}

\newcommand{\StSp}{\mbox{\rm StSp}}

\newcommand{\ld}{\langle} \newcommand{\rd}{\rangle}

\newcommand{\SK}{\mathrm{SK_1}}
\newcommand{\KSp}{\mathrm{K_1Sp}}

\def\rmk{\refstepcounter{theorem}\paragraph{{\bf Remark} \thetheorem}}
\def\proof{\paragraph{Proof}}

\def\notation{\paragraph{\bf Notation}}

\newcommand{\remark}{\rmk}

\numberwithin{equation}{subsection}

\begin{document}	

\title{ Improved injective stability for relative $\mathrm{K_1Sp}$-groups}

\author{Sourjya Banerjee}
\address{Sourjya Banerjee, Department of Mathematics, Indian Institute of Technology Hyderabad,
 Kandi, Sangareddy, Telangana 502284, India}
\email{sourjya@math.iith.ac.in, sourjya91@gmail.com}

\author{Kuntal Chakraborty}
\address{Kuntal Chakraborty, Stat-Math Unit, Indian Statistical Institute, 203, B.T. Road, Kolkata - 700108}
\email{kuntal.math@gmail.com}


\begin{abstract} 
We prove a relative version of Vorst’s theorem concerning the equality of the group of all invertible matrices and the group of all elementary matrices over $R[X]$ with respect to an ideal $I\subset R$ such that $R/I$ is regular, where $R$ is a regular $k$-spot. We then introduce the relative elementary symplectic Witt group and show that it fits into a relative version of the Karoubi periodicity sequence. Combining these results, we improve the existing injective stability bounds for relative linear and symplectic $\rm K_1$-groups of smooth affine algebras over various base fields. As an application, we obtain a necessary and sufficient criterion for the freeness of stably free modules over smooth real $4$-folds with empty real locus.

\end{abstract}

\maketitle

\begin{center}
    Dedicated to Professor Ravi A. Rao on his seventieth birthday.
\end{center}
\noindent\rule{\textwidth}{0.5pt}

{\tableofcontents}

	\section{Introduction}

The study of \emph{injective stability for $\mathrm{K}_1(R)$} originates from the seminal work of Bass--Milnor--Serre~\cite{bms}, where the stabilization of the sequence of unstable pointed sets
\begin{equation}\label{eq:stabilization}
	\dots \xrightarrow{\Gamma_{n-1}}
	\GL_n(R)/\E_n(R)
	\xrightarrow{\Gamma_n}
	\GL_{n+1}(R)/\E_{n+1}(R)
	\xrightarrow{\Gamma_{n+1}} \dots
\end{equation}
was investigated. They showed that the sequence in \ref{eq:stabilization} stabilizes to the Whitehead group $\mathrm{K}_1(R)$. The injective stability problem asks for the least integer $n$ such that $\Gamma_{n+i}$ is injective for all $i \ge 0$. In~\cite{bms}, a bound of $\max\{3,d+3\}$ was established, where $d=\dim(R)$, and it was conjectured that the bound could be improved to $\max\{3,d+2\}$. This conjecture was later confirmed by Vaser\v{s}te\u{\i}n~\cite{VasersteinBMS}, who also extended the result to relative $\mathrm{K}_1$-groups as well as to relative symplectic $\mathrm{K}_1$-groups with respect to an arbitrary ideal $I\subset R$. Vaser\v{s}te\u{\i}n's bounds are known to be optimal in general.

In~\cite{Suslin1984}, Suslin proved that stably free $R$-modules of rank $\dim(R)$ are free when $R$ is a regular affine domain over a perfect $C_1$ field (the condition on the base field is actually more general, see the condition in \ref{Rao-vdK}). Parimala subsequently observed that the regularity assumption on $R$ can be removed from Suslin's theorem. Building on this, Rao--van~der~Kallen~\cite{rvdk} established a $\mathrm{K}_1$-theoretic analogue using the idea of excision algebras, where Parimala's observation plays a crucial role. In particular, they proved that if $R$ is a nonsingular affine algebra of dimension $d$ over a perfect $C_1$ field, then the injective stability bound for the $\mathrm{K}_1(R)$ improves to $d+1$. Following this philosophy, Basu, Chattopadhyay, and Rao~\cite{br,bcr} further investigated and refined injective stability bounds for $\KSp$-groups. Subsequently, Gupta in \cite{AG1} and the first author in \cite{Banerjee2025} studied the relative case with respect to principal ideals, for both linear and symplectic $\mathrm{K}_1$-groups.

Despite this progress, injective stability for relative linear and symplectic $\mathrm{K}_1$-groups with respect to an {arbitrary} ideal has remained open, a gap explicitly highlighted by Gupta in~\cite[Question~4.9]{AG1}.
In this article, we address this problem for relative symplectic (and linear) $\mathrm{K}_1$-groups with respect to a smooth ideal $I$, i.e., when the coordinate ring $R/I$ of the variety $V(I)$ is regular.  The main theorem of the article is the following (see Section~\ref{laurent}).

\bt\label{mt}
Let $X=\Spec(R)$ be a smooth affine $k$-variety of dimension $d$, where $k$ is an infinite field such that $\rm{char}(k)\not=2$. Let $I\subset R$ be a smooth ideal of $R$. Then the following hold.
	\begin{compactenum}
	\item If $d\ge 5$ is odd, then the injective stability bound for $\mathrm{K_1Sp}(R,I)$ improves to $d+1$.
	\item If $d\ge 6$ is even, and $k$ is a perfect $C_1$ field, then the injective stability bound for $\mathrm{K_1Sp}(R,I)$ improves to $d$.
    \item  If $d\ge 6$ is even, and $k=\mathbb R$ and the set of all real points $X(\mathbb R)=\emptyset$, then the injective stability bound for $\mathrm{K_1Sp}(R,I)$ improves to $d$.
	\item If $d\ge 7$ is odd, and $k$ is an algebraically closed field such that $\frac{1}{d!}\in k$, then the injective stability bound for $\mathrm{K_1Sp}(R,I)$ improves to $d-1$.
    \item If $d\geq 5$ is an odd integer, then the injective stability bound for $\mathrm{K_1Sp}(A,J)$ improves to $d+1$, where $A=R[T,T^{-1}]\text{ and }J=I[T,T^{-1}]$.
\end{compactenum}
\et

Previously, Theorem \ref{mt}(1) was established in the absolute case, i.e., when $I=R$, by Chattopadhyay--Rao in \cite{cr1}. Under the additional assumption that $d \equiv 0 \pmod{4}$, assertion (2) was proved in the absolute case by Basu--Chattopadhyay--Rao in \cite{bcr}, while assertion (3) was proved by the first author in \cite{Banerjee2025} when $I$ is a principal ideal. Finally, assertion (4) was proved by Gupta in \cite{AG1} under the additional assumptions that $d \equiv 1 \pmod{4}$ and $I$ is a principal ideal.
 
To treat the linear and symplectic cases uniformly, we write $\mathrm{S}(n, R, I)$ 
to denote either $\SL_n(R, I)$ or $\Sp_{2m}(R, I)$ when $n = 2m$, and similarly $\mathrm{E}(n, R, I)$ for the corresponding elementary subgroups. Our initial approach to tackling arbitrary ideals builds on a simple observation that any matrix in $\mathrm{S}(n, R, I)$ can be lifted to a matrix 
in $\mathrm{S}(n, R \oplus I,0\oplus I)$ [cf. \ref{general lift}]. Since the canonical projection map $R \oplus I 
\twoheadrightarrow R$ splits, Suslin's retract lemma~(\ref{K_2-lemma}) identifies the relative group $\E(n, R \oplus I, 0 \oplus I)$ in terms of 
the absolute group $\E(n, R \oplus I)$. On the other hand, the map $
R \oplus I \to R$ defined by $(r, i) \mapsto r + i$ induces a homomorphism
$ \E(n, R \oplus I,\, 0 \oplus I) \to \E(n, R, I)$, thereby providing 
a bridge between the groups $\mathrm{E}(n,R,I)$ and $\mathrm{E}(n,R \oplus I,0\oplus I)$.

\subsection{A relative version of Vorst's theorem}A key observation in the study of injective stability for linear and symplectic $\mathrm{K}_1$-groups of regular affine algebras is that this phenomenon relies heavily on Vorst’s theorem~\cite[Theorem~3.3]{Vorst} (see also~\cite{br}). Vorst's theorem essentially asserts that if $R$ is a regular local ring essentially of finite type over a field, then
\[
\mathrm{S}(n,R[X]) = \E(n, R[X]).
\]
 However, even when $R$ is regular, the excision algebra $R \oplus I$ need not
remain regular. Thus, the interplay between the groups $\mathrm{E}(n,R,I)$ and $\mathrm{E}(n,R \oplus I,0\oplus I)$ described above is not sufficient to establish a relative version of Vorst's theorem. In fact, the absence of a relative version of Vorst's theorem turns out to be the first genuine obstruction to establishing stability results for relative symplectic $\mathrm{K}_1$-groups. In Section~\ref{Vorst}, we establish the following relative version of Vorst's theorem for a smooth ideal $I$. This constitutes the technical heart of the paper, which is perhaps of independent interest.
\bt\label{rvthorem}
Let $R$ be a regular $k$-spot over an infinite field $k$, i.e., $R=A_{\mathfrak{p}}$, where $A$ is an affine $k$-algebra and $\mathfrak{p}\subset A$ is a regular prime ideal. Let $I \subset R$ be a smooth ideal. Then
\[
\mathrm{S}(n,R[X],I[X]) = \E(n,R[X],I[X]),
\]where in the linear case $n\ge 3$, and in the symplectic case $n\ge 6$.
\et

Our approach in the relative setting differs substantially from that of Vorst. Instead, it is inspired by Lindel’s celebrated work~\cite{Lindel81} on the Bass--Quillen conjecture. More precisely, we use Nashier’s finer refinement~\cite{budh} of Lindel’s philosophy, that a regular $k$-spot may be viewed as an étale extension of a localization of a polynomial ring over a field. This allows us to reduce the problem from a regular $k$-spot to a localization of a polynomial ring, with a suitable choice of indeterminates.

\subsection{Symplectic completion of relative unimodular rows} 
 In~\cite{rvdk}, Rao--van der Kallen observed that, for a regular $k$-algebra $R$ of dimension $d\ge 3$, the triviality of $\frac{\Um_{d+1}(R \oplus \langle a \rangle)}{\SL_{d+1}(R \oplus \langle a \rangle)}$
 for every $a \in R$  is sufficient for the 
 injective stability bound of $\mathrm{K}_1(R)$ to become $d + 1$. We show in Proposition~\ref{cis} that this philosophy also extends to the setting of relative symplectic groups. Consequently, the problem of completion of a relative unimodular rows with respect to a relative symplectic matrix arises naturally in our setting. However, in the relative setting, the necessary completion results were not previously available in the required generality.
In Theorem \ref{symplectic completion}, we prove the following result (for the definitions of $\A_{2n}(R,I)$ and $\Sp_{\varphi}(R,I)$ we refer to \ref{relative witt gr} and \ref{symplectic matrix w.r.t. a form}).

\bt\label{sc}
Let $X=\Spec(R)$ be a smooth affine $k$-variety of dimension $d$, where $k$ is an infinite field such that $\operatorname{char}(k)\not= 2$. Let $I$ be a smooth ideal of $R$.  

\begin{compactenum}
    \item Let $d\ge 3$ be an odd integer and let $\varphi\in \A_{d+1}(R,I)$. Assume that one of the following holds. 
    \begin{compactenum}
        \item If $k$ is a perfect $C_1$ field. 
        \item If $k=\mathbb{R}$ and $X(\mathbb{R})=\emptyset$.
    \end{compactenum}Then $\Um_{d+1}(R,I)= e_1 \Sp_{\varphi}(R,I)$.
    \item Let $d\ge 4$ be an even integer and let $\varphi\in \A_{d}(R,I)$. If $k=\ol k$ such that $\frac{1}{(d-1)!}\in k$, then $\Um_d(R,I)=e_1\Sp_{\varphi}(R,I)$.
    
\end{compactenum}
\et 

Whenever $\varphi$ is the standard elementary alternating form $\chi_{n}$, the group $\Sp_{\varphi}(R,I)$ coincides with the usual relative symplectic group $\Sp_{2n}(R,I)$. Prior to us, Theorem~\ref{sc} part (1)(a) was proved by Basu--Chattopadhyay--Rao~\cite{bcr} under the additional assumptions that $\varphi=\chi_{n}$ and $d \equiv 1 \pmod{4}$. Under the same assumptions, part~(1)(b) was established by the first author \cite{Banerjee2025}. Part~(2) was proved by Gupta~\cite{AG1} under the additional assumptions that $\varphi=\chi_{n}$ and $d \equiv 2 \pmod{4}$. Recently, Syed showed in a series of papers~\cite{Syed21,Syed2024,Syed25} how to remove the restriction on $d$ and settle the remaining cases in the absolute setting using sophisticated cohomological methods. One of the central motivations of his works is the computation of the symplectic elementary Witt group $W_E(R)$ in \cite{FRS}, where the authors identify $W_E(R)$ with the Grothendieck--Witt group via an exact sequence known as the Karoubi periodicity sequence.

In one of the preparatory steps in establishing Theorem \ref{sc} we introduce the relative elementary symplectic Witt group $W_E(R,I)$ in Section~\ref{relative elementary symplectic Witt groups}. We further show that $W_E(R,I)$ fits into a relative version of the Karoubi periodicity sequence.
%
%
%
%

\subsection{Application: stably free modules over smooth real $4$-folds}

In Section~\ref{application}, we use the theory developed in this article and \cite[Theorem 4]{bfl} to prove the following analogue of Syed's result \cite[Theorem 3.19]{Syed21} for smooth real $4$-folds with empty real locus.
\begin{theorem}
	Let $X=\Spec(R)$ be a smooth real variety of dimension $4$ such that $X(\mathbb{R})=\emptyset$. Then stably free $R$-modules are free if and only if
$	\widetilde{V_{SL}}(R)=0$, where $\widetilde{V_{SL}}(R)$ is a quotient of the generalized Vaserstein symbol $\widetilde{V}(R)$, as defined in \cite{Tariq2019}.
\end{theorem}

\subsection{Layout of the article} 

The article is organized as follows. Section~\ref{2} recalls requisite background on linear and symplectic groups, analytic isomorphisms, and excision algebras. Section~\ref{K2} defines the relative symplectic Steinberg group and proves a symplectic $\mathrm{K}_2$-lemma. Section~\ref{Vorst} establishes a relative version of Vorst's theorem. Section~\ref{affine algebra} addresses injective stability for relative linear $\mathrm{K}_1$-groups. We define relative elementary symplectic Witt groups in Section~\ref{relative elementary symplectic Witt groups} and study symplectic completions of relative unimodular rows in Section~\ref{symplectic completion of unimodular rows}. Section~\ref{laurent} contains the main result of the paper, establishing injective stability bounds for relative symplectic groups of various affine algebras. Finally, Section~\ref{application} provides an application.
\subsection{Convention} Unless otherwise stated, all rings considered in this article are assumed to be commutative Noetherian with $1(\not=0$). We denote by $e_i$ the canonical basis vector of $R^n$ whose only nonzero entry is $1$ in the $i$-th coordinate. 

\addtocontents{toc}{\protect\setcounter{tocdepth}{-1}}

\section*{Acknowledgment} The authors sincerely thank Tariq Syed for helping them to understand his results \cite[Theorems~3.6 and~3.8]{Syed2024}. Part of this work was carried out during the second author's visit to IMSc, and he is grateful to Rahul Gupta and IMSc for their hospitality. The second author acknowledges the Department of Science and Technology (DST), India, for the INDO-RUSS project (DST/INT/RUS/RSF/P-48/2021).

\addtocontents{toc}{\protect\setcounter{tocdepth}{1}}

	\section{Preliminaries}\label{2}
	
	\subsection{Some assorted results}
	
	This section summarizes several results and definitions from the literature that are frequently used in this article to prove the main theorem. We may restate or improve these results as necessary. Before proceeding further, we recall several definitions related to the
	relative symplectic groups from \cite[Chapter 1]{SV}. We believe most of the following definitions are well-known. However, for the sake of completeness, we include them.
	
    \bd
		Let $R$ be a ring and $I$ be an ideal of $R$.
		
		\begin{compactenum}
			\item Let $\M_n(R,I)$ and $\GL_n(R,I)$ be the kernels of the canonical maps $\M_n(R)\longrightarrow \M_n(R/I)$ and $\GL_{n}(R) \longrightarrow \GL_{n}(R/I)$ respectively. The group $\SL_n(R,I)$ denotes the subgroup of $\GL_n(R,I)$ consisting of elements of determinant $1$.
			
			\item  Let $i\neq j$ and let $e_{ij}(\lambda)$ be the matrix whose all diagonal entries are $1$ and $ij^{th}$ entry is $\lambda$ and the rest entries are zero. The elementary group $\E_n(R)$ is the subgroup of $\GL_n(R)$ generated by the matrices of the form $e_{ij}(\lambda)$, where $\lambda\in R$ and $i\neq j$.
			
			\item  The relative elementary group is the subgroup of $\SL_{n}(R,I)$ generated by the matrices of the form $\alpha e_{i,j}(a) \alpha^{-1}$, where $\alpha \in \E_{n}(R),i\neq j$ and $a\in I$.
			
			\item If $\alpha\in \M_m(R)$ and $\beta\in \M_n(R)$, then by $\alpha\perp \beta$ we denote the matrix $$\begin{pmatrix}
				\alpha &0\\0&\beta
			\end{pmatrix}\in \M_{m+n}(R).$$
			
			\item 	We identify $\GL_{n}(R)$ with a subgroup of $\GL_{n+1}(R)$ by associating a matrix $\alpha\in \GL_n(R)$ with  the matrix $1\perp \alpha \in \GL_{n+1}(R)$.
			We set the following:
			
            \begin{compactenum}[\quad\quad\quad\quad]
	\item 			$\GL(R):=\varinjlim \GL_{n}(R), \, \GL(R,I):=\varinjlim \GL_{n}(R,I)$;
		\item
				$\SL(R):=\varinjlim \SL_{n}(R), \, \SL(R,I):=\varinjlim \SL_{n}(R,I)$;
			\item	$\E(R):=\varinjlim \E_{n}(R), \, \E(R,I):=\varinjlim \E_{n}(R,I)$;
			\end{compactenum}

			\item The matrix $\chi_n\in \E_{2n}(R)$ is defined inductively by $\chi_1=\begin{pmatrix}
				0&1\\-1&0
			\end{pmatrix}$ and $\chi_{n+1}=\chi_n\perp \chi_1$ for $n\geq 1$. 
			
			\item 
			The matrix $\sigma_{n}\in \GL_{2n}(R)$ is defined inductively by the following formula:
            
            $\sigma_1=\begin{pmatrix}
				0&1\\1&0
			\end{pmatrix}$ and $\sigma_{n+1}=\sigma_n\perp \sigma_1$ for $n\geq 1$.

				\item A matrix $\alpha \in \M_n(R)$ is said to be alternating if $\alpha=\beta-\beta^T$ for some $\beta\in \M_n(R)$, i.e., $\alpha$ is skew-symmetric and all the diagonal entries are $0$.
			
			\item The symplectic group $\Sp_{2n}(R)$ is defined by $\{\alpha\in \GL_{2n}(R): \alpha^T\chi_n \alpha =\chi_n\}$.
			
			\item Let $\sigma:\{1,\dots, 2n\}\to \{1,\dots,2n\}$ be the permutation defined by $\sigma(2i)=2i-1$ and $\sigma(2i-1)=2i$ for all $i=1,\dots,n$. Then the elementary symplectic group $\ESp_{2n}(R)$ is a subgroup of $\Sp_{2n}(R)$ generated by the matrices 
			
			\[ se_{ij}(\lambda)=  \left\{
			\begin{array}{ll}
				{\rm I}_{2n}+\lambda E_{ij}, & \text{ if } i=\sigma(j) \\
				{\rm I}_{2n}+ \lambda E_{ij}-(-1)^{i+j} \lambda E_{\sigma(j)\sigma(i)}, & \text{ if } i\neq \sigma(j)
							\end{array} 
			\right. \]where $E_{ij}$ is the matrix whose $ij$-th entry is $1$ and all the other entries are $0$, and $\lambda\in R$. We denote $\ESp_{2n}(I)$ by the subgroup of $\Sp_{2n}(R)$ generated by $se_{ij}(\lambda)$, where $1\leq i\ne j\le 2n$ and $\lambda\in I$.
			
			\item The relative symplectic group $\Sp_{2n}(R,I)$ is denoted by the set $\{\alpha\in \GL_{2n}(R,I): \alpha^T \chi_n \alpha=\chi_n\}$.  The relative elementary symplectic group $\ESp_{2n}(R,I)$ is the subgroup of $\Sp_{2n}(R,I)$ generated by the matrices of the form $\alpha se_{ij}(x) \alpha^{-1}$, where  $\alpha\in \ESp_{2n}(R)$, $i\neq j$ and $x\in I$.                                                                              
			\item One can identify $\Sp_{2n}(R)$ with a subgroup  of $\Sp_{2n+2}(R)$ via the map $\alpha\mapsto
			\rm I_2 \perp \alpha\in \Sp_{2n+2}(R)$. We set the following:
				\begin{compactenum}
				
                \item $\Sp(R,I):=\varinjlim \Sp_{2n}(R,I)$; 
				\item $\ESp(I):= \varinjlim \ESp_{2n}(I)$;  
                \item $\ESp(R,I):=\varinjlim \ESp_{2n}(R,I)$;
                \item $\Sp(R):=\Sp(R,R)$;
                \item $\ESp(R):=\ESp(R,R)$.
				\end{compactenum}
			\item A row $(a_1,a_2,\dots,a_n)\in R^n$ is said to be unimodular if there exist $(b_1,b_2,\dots,b_n)\in R^n$ such that $\sum_{k=1}^{n}a_kb_k=1$. A row $(a_1,a_2,\dots,a_n)\in R^n$ is said to be relative unimodular row with respect to the ideal $I$ if it is unimodular and $$(a_1,a_2,\dots,a_n)\equiv (1,0,\dots,0) \pmod I,$$ i.e., $a_1-1,a_2,\dots,a_n$  all belong to $I$. The set of all relative unimodular rows with respect to the ideal $I$ is denoted by $\Um_n(R,I)$. By an abuse of notation, we write $\Um_{n}(R)=\Um_n(R,R)$.

		\end{compactenum}

	\ed

    \notation\label{uniform def.} Let $R$ be a ring and let $I \subset R$ be an ideal. We denote $\SL_n(R,I)$ or $\Sp_{2m}(R,I)$ when $n=2m$, by $\mathrm{S}(n,R,I)$. Similarly, we denote $\E_n(R,I)$ or $\ESp_{2m}(R,I)$ when $n=2m$, by $\E(n,R,I)$. If $I=R$, then by an abuse of notation we set $\mathrm{S}(n, R,I)$ and $\E(n,R,R)$ by  $\mathrm{S}(n, R)$ and $\E(n,R)$, respectively. For any $\lambda\in R$, we denote $\eta_{ij}(\lambda)$ as an elementary generator of $\E(n,R)$, i.e., when $\E(n,R)= \E_n(R)$, then $\eta_{ij}(\lambda)=e_{ij}(\lambda)$ for $1\leq i\neq j\leq n$ and when $\E(n,R)= \ESp_{2m}(R)$, then $\eta_{ij}(\lambda)= se_{ij}(\lambda)$ for $1\leq i\neq j\leq 2m$.
	Let $G$ be a group and $H,K$ be two subgroups of $G$. By $[H,K]$, we denote the subgroup of $G$ generated by the commutators $[h,k]:=hkh^{-1}k^{-1}$, where $h\in H$ and $k\in K$.

We recall the following result from \cite[Chapter V]{Bass1968} and \cite{Vaserstein1970}.
\bl
Let $R$ be a commutative ring and $I$ be an ideal of $R$. Then we have the following.
\begin{compactenum}[\quad\quad\quad\quad(1)]
	\item $\E(R)=[\E(R), \E(R)]=[\GL(R), \GL(R)]$;
	\item $\E(R,I)= [\E(R), \E(R,I)]=[\GL(R), \GL(R,I)]$;
	\item $\ESp(R,I)= [\ESp(R),\ESp(I)]=[\Sp(R),\Sp(R,I)]$.
\end{compactenum}
\el

\bd
Let $R$ be a commutative ring and $I$ be an ideal of $R$. The linear and symplectic $\mathrm{K_1}$-groups are defined as follows.
\begin{compactenum}[\quad\quad(1)]
	\item $\mathrm{K_1}(R,I):=\GL(R,I)/\E(R,I)$,
    \item $ \SK(R,I):=\SL(R,I)/\E(R,I)$;
	\item $\KSp(R,I):=\Sp(R,I)/\ESp(R,I)$;
	\item $\mathrm{K_1}(R):=\mathrm{K_1}(R,R)$; 
    \item $\SK(R):=\SK(R,R)$;
    \item $ \KSp(R):=\KSp (R,R)$.
\end{compactenum}

\ed
	
	%
	%
		%
	%

\bd
		A ring homomorphism $\phi: B\to D$ is said to be a retract if there exists $\gamma: D\hookrightarrow B$ such that $\phi\circ \gamma$ is identity on $D$. We also say $D$ is a retract of $B$.
	\ed

	\remark\label{rectract}
	
		Let $A$ be a ring and $I\subset A$ be an ideal. Then the canonical projection map $\pi:A\oplus I\twoheadrightarrow A$ is a retract.

	The next result is a special case of Suslin's retract lemma. The proof of the following version can be found in \cite[Lemma 3.3]{acr}.
	
	\bl
	
		\label{K_2-lemma}
		Let $\pi: B\twoheadrightarrow D$ be a retract and $J$ be the kernel of $\pi$. Then for $n\geq 2$,  $$\E(n,B,J)=\mathrm{S}(n,B,J)\cap \E(n,B).$$
	\el

%
%
%

	%
	
	%
\bd A ring $R$ is said to be essentially of finite type over a field $k$, if $R$ is of the form $S^{-1}A$, where $A$ is an affine $k$-algebra and $S\subset A$ is a multiplicatively closed set.
\ed
\bd
	
		Let $k$ be a field. Let $C$ be an affine $k$-algebra and let $\mathfrak{p}$ be a prime ideal of $C$. The localization ring $C_{\mathfrak{p}}$ is said to be a regular $k$-spot if $C_{\mathfrak{p}}$ is a regular local ring.
\ed
	
The next result is due to Vorst \cite[Theorem 3.3]{Vorst}. In the symplectic case it was proved by Basu–Rao \cite[Theorem 3.8]{br}, where they established it under the additional assumption that the base field is perfect. This assumption can, however, be removed using an argument of Mohan Kumar given in \cite[Theorem 3.3, last paragraph]{Vorst}. For completeness, we include Mohan Kumar’s argument in the proof.
	
	\bl
		\label{Symplectic K_1}
		Let $k$ be a field, and let $R$ be a regular $k$-spot. Then we have $$\mathrm{S}({n},R[X])=\E({n},R[X]),$$where $n\geq 3$ in the linear case and $n\geq 6$ in the symplectic case.
	\el
	
	 \proof 
    
    Clearing denominators, one can find finitely many elements $\alpha_1,\dots, \alpha_s\in k$ such that (1) a set of generators of $I$ is in $A:= \mathbb{F}_p[\alpha_1,\dots,\alpha_s][X_1,\dots,X_t]$, (2) a minimal set of generators of $\mathfrak{p}$ is in $A$, and (3) $\alpha\in (A/J)_{\mq}[X]$, where $J\subset A$ is the ideal generated by the set of generators of $I$, and $\mq\subset A$ is the ideal generated by the set of generators of $\mathfrak{p}$. If we take $B:=(A/J)_{\mq}$, then $B \subset R$ is a regular $\mathbb{F}_p$-spot such that $\alpha \in \mathrm{S}({n},B[X])$ (cf. \cite[Theorem 3.13, last paragraph]{Chakraborty2023}). Since $\mathbb{F}_p$ is a perfect field one can apply \cite[Theorem 3.3]{Vorst} or \cite[Theorem 3.8]{br} to obtain that $\alpha \in \E({n},B[X])\subset \E(n,R[X])$. This concludes the proof. \qed
	
One can find the following result in \cite[Proposition 2.2.4]{BH}.
	
	\bl
		\label{quotient of regular is regular}
		Let $R$ be a regular local ring  and $I$ be an ideal of $R$. Then $R/I$ is regular if and only if $I$ is generated by a subset of a regular system of parameters.
	\el

	\bl
		\label{Nagata}
		$($Nagata$)$
		Let $k$ be a field. Let $f$ be a polynomial in $k[X_1,X_2,\dots,X_d]$ and $\phi(X_1)\in k[X_1]$ be a monic polynomial. Then there exists a change of variables $X_1\mapsto X_1$ and $X_i\mapsto X_i+\phi(X_1)^{r_i}$, $2\leq i\leq d$, such that $f=c\cdot h(X_1,X_2,\dots,X_d)$, where $c\in k\setminus 0$ and $h\in k[X_2,X_3,\dots,X_d][X_1]$ is a monic polynomial in $X_1$ .
	\el
	We conclude this subsection with a result due to Ojanguren \cite{ojanguren}.
	
	\bl
		\label{ojanguren}
		Let $k$ be an infinite field, and let $\varphi(X_1)\in k[X_1]$ be an irreducible monic polynomial. Let $F, G \in k[X_1,\dots, X_d]$ be
		two polynomials having no common factors. Let $\mathfrak{m}=\langle\varphi(X_1),X_2, \dots, X_d\rangle$ be such that $F\notin \mathfrak{m}$.
		Then there exists a change of variables of the type $X_1\mapsto X_1$ and $X_i\mapsto X_i+ \alpha_i \varphi(X_1)$, $i\ge 2$; for some suitably chosen $\alpha_i\in k$, so that
		\begin{compactenum}
		\item $F, G$ are monic in $X_1$ and		
		\item  $F(X_1,0,\dots, 0)$ and $G(X_1,0,\dots, 0)$ are coprime.
			\end{compactenum}
	\el

\subsection{Analytic isomorphism}
This subsection recalls some basic facts about analytic isomorphisms, which play a crucial role throughout the article.
	
	\bd
		Let $A, B$ be two rings and let $\phi: B\rightarrow A$ be a ring homomorphism. For an element $h\in B$, the homomorphism $\phi$ is said to be analytic  isomorphism along $h$, if the following conditions are satisfied:
			\begin{compactenum}
		\item  the element $h$ is a non-zero divisor in $B$,
		\item the element $\phi(h)$ is a non-zero divisor in $A$, and
		\item the map $\phi$ induces an isomorphism between $B/\langle h\rangle B$ and $A/\langle \phi(h)\rangle A$.
			\end{compactenum}
	\ed
	
	\begin{example}
		\label{example of a.i.}
	Let $R$ be a ring and let $s,t \in R$ be such that $Rs+Rt=R$. Assume that $s$ is a non-zero divisor in $R$. Then, as shown in \cite[Example~1]{Bhatwadekar1983}, the localization map $R \to R_t$ is an analytic isomorphism along $s$.
	\end{example}

	\bd
		Let $(R,\mathfrak{m})$ be a local ring. A monic polynomial $f\in R[X]$ is said to be a Weierstrass polynomial if $f= X^n+ a_1X^{n-1}+\dots+a_n$, where $ a_i \in \mathfrak{m}$  for $i=1,2,\dots,n$.
	\ed
	\smallskip
	
	\bd
		Let $(R,\mathfrak{m})$ be a local ring and $\varphi$ be a  monic polynomial in $R[X]$. A polynomial $h\in R[X]$ is said to be a Weierstrass polynomial relative to $\varphi$ if $h=\varphi^n+ a_1\varphi^{n-1}+\dots +a_n$, where $a_i\in\mathfrak{m}$ for all $i=1,2,\dots, n$.
	\ed
	
	We note the following results from \cite[Proposition 1.7 and Theorem 2.8]{budh}.
	
	\bl
		\label{0.6}
		Let $(R,\mathfrak{m})$ be a local ring and $f\in R[X]$ be a   Weierstrass polynomial. Then $R[X] \hookrightarrow R[X]_{\langle \mathfrak{m},X\rangle}$ is an analytic isomorphism along $f$.
	\el
	
	\bt
		\label{Nashier}
		Let $(R,\mathfrak{m})$ be a regular local algebra of dimension $d$ with a separating ground field $K$. Let $g$ be any element of $\mathfrak{m}^2$.
		Then there exists a regular local subring $S$ of $R$ such that:
		\begin{compactenum}[\quad\quad (i)]
			\item $S = K[X_1,\dots, X_d]_{\langle \varphi(X_1),X_2,\dots,X_d)\rangle}$, where $\varphi(X_1) \in K[X_1]$ is an irreducible monic polynomial.
			\item $S\subset R$ is analytically isomorphic along $h$, for some $h \in gR\cap S$ (here $h$ depends on the choice of $g$).
		
		\end{compactenum}

	\et

	The following result is taken from \cite[Theorem 1.2]{Rao2} which is a finer version of the above result.
	
	\bt\label{NRao}
		Let $(R,\mathfrak{m})$ be a regular $k$-spot of dimension $d$ over a perfect field $k$. Let $g\in \mathfrak{m}^2$ and let $\{g,f_1,f_2,\dots, f_{d-1}\}$ be a sequence in $R$ with $\{f_1,f_2,\dots, f_{d-1}\}$ is a part of a minimal set of generators of $\mathfrak{m}/\mathfrak{m}^2$. Then there exists a field $K\supseteq k$ and a regular $K$-spot $R'$ such that 
		\begin{compactenum}
			\item $R'=K[X_1,X_2,\dots, X_d]_{\langle X_1,X_2,\dots, X_{d-1}, \varphi(X_d)\rangle}$, where $\varphi(X_d)\in K[X_d]$ is an irreducible monic polynomial. Moreover, one may assume that $X_i=f_i$ for $i=1,\dots, d-1$.
			\item $R'\subset R$ is analytically isomorphic along $h$ for some $h\in gR\cap R'$.
		\end{compactenum}
	\et

	We note the following results from \cite[Lemmas 2.5 and 3.6]{Rao2}.

	\bl
	\label{coprime_imply_comaximal}
	Let $(R,\mathfrak{m})$ be a local ring. Let $F, G\in R[X]$, with either
	$F$ or $G$ a monic polynomial. Assume that $\ol{F}$, $\ol{G}$ are coprime where `bar'
	denotes `modulo $\mathfrak{m}$'. Then $F$ and $G$ are comaximal.
	\el
	
	\bl
		\label{relative analytic}
		Let $(R,\mathfrak{m})$ be a local ring and $\varphi(X)$ be a monic polynomial in $R[X]$. Let $h$ be a Weierstrass polynomial relative to $\varphi$. Then the inclusion map $R[X]\hookrightarrow R[X]_{(\mathfrak{m},\varphi(X))}$ is an analytic isomorphism along $h$.
	\el

	The next results can be found in \cite[Lemma 2.8 and Corollary 2.9]{kcr}.
	
	\bp
		\label{analytic variation}
		Let $(A,\mathfrak{m})$ be a local ring and let $R=A[X]$ be a
		polynomial ring. Suppose $h\in \langle \mathfrak{m},X\rangle\subset R$ such that $h$ is comaximal with any element in $R\setminus \langle \mathfrak{m},X\rangle $. Then $R\hookrightarrow R_{\langle \mathfrak{m}, X\rangle }$ is
		an analytic isomorphism along $h$.
	\ep

	\bc
		\label{pseudo-weierstrass}
		Let $(A,\mathfrak{m})$ be a local ring and $R=A[X]$ be a polynomial ring. Let $\varphi$ be a monic irreducible polynomial in $R$. Suppose $h\in R$ such that $h\in \langle \varphi,\mathfrak{m}\rangle $ and $h$ is comaximal to every element in $R\setminus \langle \varphi,\mathfrak{m}\rangle $. Then the inclusion $R\hookrightarrow R_{\langle \varphi,\mathfrak{m}\rangle }$ is an analytic isomorphism along $h$.
	\ec
	
\subsection{Excision algebras} In this subsection, we recall the excision algebra and some related results.
\bd

Let $R$ be a ring and let $I$ be an ideal of $R$. The excision algebra of $R$ with respect to the ideal $I$ is defined as
\[
R \oplus I := \{(r,i) \mid r \in R,\ i \in I\},
\]
where addition is defined componentwise and multiplication is defined by
$(r,i)(s,j) := (rs,\, rj + si + ij)$.

\ed



We note the following result from \cite[Proposition 3.1]{keshari}.

\bl

\label{keshari}
Let $R$ be a ring of dimension $d$ and $I\subset R$ be an ideal of $R$. Then the excision algebra $R\oplus I$ is an $R$-algebra of dimension $d$.
\el

One can reinterpret the excision algebra $R\oplus I$ as a fiber product of rings $R\times_{R/I} R$. For the definition of the fiber product, we refer to \cite{Milnor1971}. We recall the following definition of the double ring.
 
 \bd
 Consider the fibre product
 \[
 \begin{tikzcd}
 	D(R,I) \arrow{r}{p_1} \arrow{d}{p_2}
 	&R \arrow{d}{\pi}\\
 	R \arrow{r}{\pi} &R/I
 \end{tikzcd}
 \]
 Here the ring $D(R,I)$ is called the double ring of the ring $R$ with respect to the ideal $I$. The ring $D(R,I)$ can be identified with the set $\{(a,b)\in R\times R:a-b\in I\}$.
 \ed
 
 
 \bl
 Let $R$ be a ring and $I$ be an ideal of $R$. Then $D(R,I)$ is canonically isomorphic to the excision algebra $R\oplus I$. 
 \el
 
 \proof One may check that the map
\[
u : R \oplus I \longrightarrow D(R, I),
\]
defined by $u(a, i) = (a, a + i)$, is an isomorphism, and its inverse is given by
\[
v : D(R, I) \longrightarrow R \oplus I,
\]
defined by $v(x, y) = (x, y - x)$. This concludes the proof.\qed

\rmk Any matrix $\alpha \in \M_n(R)$ can be identified with an element of $\M_n(R \oplus I)$ via the canonical map $i : R \to R \oplus I$ sending $a \mapsto (a,0)$. We refer to the element $(a,0)\in R\oplus I$ as the canonical lift of $\alpha$. With a slight abuse of notation, we use the same notation to denote the canonical lifts of $\mathrm{I}_n$, $\chi_n$, and $\sigma_n$ in $\M_n(R \oplus I)$.



The following two definitions play a crucial role in this article.

\bd
\label{lift of unimodular row}
Let $v\in \Um_n(R,I)$. One can write $v= (1+v_1, v_2,\dots, v_n)$, where $v_i\in I$ for $1\leq i\leq n$. Then by \cite[Lemma 1]{Vasershtein1971} there exists $w=(1+w_1,w_2, \dots, w_n)\in \Um_n(R,I)$ such that $(1+v_1)(1+w_1)+v_1w_1+\dots+ v_n w_n=1$. From this relation, it follows that 
	$$(1,v_1)(1,w_1)+(0,v_2) (0,w_2)+\dots+ (0,v_n)(0,w_n)=(1,0)$$
	in $R\oplus I$. Hence the row $v_L:=((1,v_1), (0,v_2), \dots, (0,v_n))$ is in $ \Um_n(R\oplus I)$. We call $v_L$ the lift of $v$.
\ed

\bd
\label{general lift}
For any $\beta =(b_{ij})\in \M_{m,n}(I)$, we define the lift $\beta_L \in \M_{m,n}(R \oplus I)$ of $\beta$ to be the matrix
$\beta_L := ((0,b_{ij}))$. For any $\alpha \in \GL_n(R,I)$, one can write $\alpha = \mathrm{I}_n + \gamma$ for some $\gamma \in \M_n(I)$. We define the lift $\alpha_L \in \M_n(R \oplus I)$ of $\alpha$ by
$
\alpha_L := \mathrm{I}_n + \gamma_L$, where $\mathrm{I}_n$ is the canonical lift of $\mathrm{I}_n$ in $M_n(R\oplus I)$.

\ed

\bl
\label{lift of invertible}
Let $\alpha\in \GL_n(R,I)$. Then ${\alpha}_L\in \GL_n(R\oplus I,0\oplus I)$. Moreover, if $\alpha\in \SL_n(R,I)$, then ${\alpha_L}\in \SL_n(R\oplus I,0\oplus I)$.
\el

\proof First, we observe that, by the definition of $\alpha_L$, it follows that $\alpha_L \equiv \I_n \mod (0 \oplus I)$. Since $\alpha \in \GL_n(R,I)$, its inverse $\alpha^{-1}$ also lies in $\GL_n(R,I)$. Hence there exist $\gamma, \delta \in \M_n(I)$ such that $\alpha={\rm I}_n+\gamma$ and $\alpha^{-1}={\rm I}_n+\delta$. Since  $(0,b)(0,c)=(0,bc)$ for all $b,c\in I$, we get that ${\gamma_L}{\delta_L}= (\gamma\delta)_L$. Now
\begin{align*}
	{\alpha_L} {\alpha^{-1}_L}&= {\rm I_n}+{\gamma_L}+{\delta_L}+{\gamma_L} {\delta_L}\\
	&={\rm I}_n +(\gamma+{\delta}+{\gamma} {\delta})_L\\
	&= ({\alpha \alpha^{-1}})_L\\
	&={\rm I}_n.
\end{align*}Therefore, we obtain that ${\alpha_L}\in \GL_n(R\oplus I,0\oplus I)$.

Now suppose that $\alpha \in \SL_n(R, I)$. Expanding the determinant formula of $\alpha$ and ${\alpha}_L$, one may obtain that $\det(\alpha)=1+x$ and $\det({\alpha_L})=(1,x)$, for some $x\in I$.	Since $\alpha\in \SL_n(R,I)$, we must have $x=0$ and hence $\det({\alpha_L})=(1,0)$. In other words, the matrix ${\alpha_L}\in \SL_n(R\oplus I,0\oplus I)$. This completes the proof.\qed                                  	

\bl
\label{lifts multiplication}
Let $\beta\in \M_{2n,k}(I)$ and $\gamma\in \M_{k,2n}(I)$. Then $\chi_n {\beta_L}=({\chi_n \beta})_L$ and  ${\gamma}_L \chi_n =({\gamma \chi_n})_L$.
\el

\proof We only show that $\chi_n \beta_L = (\chi_n \beta)_L$, as the proof of the other part is identical. We prove the result by induction on $n$.

First, we assume that $n=1$. Let $\beta=\begin{pmatrix}
		b_{11}&b_{12}&\dots &b_{1k}\\b_{21}&b_{22}&\dots& b_{2k}
	\end{pmatrix}$. Then one may note that $$\chi_1 \beta=\begin{pmatrix}
		b_{21} &b_{22}&\dots& b_{2k}\\-b_{11}&-b_{12}&\dots&-b_{1k}
	\end{pmatrix}.$$ Now $(\chi_1 \beta)_L=\begin{pmatrix}
		(0,b_{21})&(0,b_{22})&\dots &(0,b_{2k})\\-(0,b_{11})&-(0,b_{12})&\dots& -(0,b_{1k})
	\end{pmatrix}=\chi_1{\beta_L}$. This concludes the proof for the case $n=1$.

We now assume that the result holds for $n = m$, that is, $\chi_m \beta_L = (\chi_m \beta)_L$ for all $\beta \in \M_{2m,k}(I)$.

For $n=m+1$, we write $\beta=\begin{pmatrix}
		A\\B
	\end{pmatrix}$, where $A\in \M_{2m,k}(I)$ and $ B\in \M_{2,k}(I)$. Then, by applying the matrix multiplication rule, one obtains the following. $$\chi_{m+1} \beta =\begin{pmatrix}
		\chi_m&0\\0&\chi_1
	\end{pmatrix} \begin{pmatrix}
		A\\B
	\end{pmatrix}=\begin{pmatrix}
		\chi_m A\\\chi_1 B
	\end{pmatrix}$$ 
Using the induction hypothesis we get the following. 	
	$$({\chi_{m+1} \beta})_L=\begin{pmatrix}
		({\chi_m A})_L\\({\chi_1 B})_L
	\end{pmatrix}=\begin{pmatrix}
		\chi_m {A_L}\\\chi_1 {B_L}
	\end{pmatrix}=\chi_{m+1} \begin{pmatrix}
		{A_L}\\{B_L}
	\end{pmatrix}=\chi_{m+1} {\beta_L}$$This concludes the proof.\qed

\bc
\label{excision of relative symplectic}
Let $\alpha\in \Sp_{2n}(R,I)$. Then ${\alpha}_L\in \Sp_{2n}(R\oplus I,0\oplus I)$.
\ec
\proof Since $\alpha\in \Sp_{2n}(R,I)$ there exists $\beta\in \M_{2n}(I)$ such that $\alpha={\rm I}_{2n}+\beta$. By Lemma \ref{lift of invertible}, we have ${\alpha_L}\in \GL_{2n}(R\oplus I,0\oplus I)$. Now applying Lemma \ref{lifts multiplication}, we get the following. 
\begin{align*}
	{\alpha}^T_L\chi_n {\alpha}_L&={\alpha^T}_L \chi_n {\alpha}_L\\
	&=({\rm I}_{2n}+\beta^T_L)\chi_n ({\rm I}_{2n}+{\beta}_L)\\
	&= \chi_n+\chi_n {\beta_L}+\beta^T_L \chi_n+\beta^T_L \chi_n{\beta}_L\\
	&=({\chi_n +\chi_n\beta+\beta^T \chi_n+\beta^T \chi_n \beta})_L\\
	&={\chi_n}\end{align*}
Hence ${\alpha}_L\in \Sp_{2n}(R\oplus I,0\oplus I)$. This completes the proof.\qed

Combining Lemmas~\ref{lift of invertible} and~\ref{excision of relative symplectic}, we obtain the following result.

\bc
\label{uniform lift}
If $\alpha\in \mathrm{S}(n,R,I)$, then ${\alpha_L}\in\mathrm{S}(n, R\oplus I,0\oplus I)$.
\ec

\remark
\label{lifts under excision}
With the notation as in Corollary~\ref{uniform lift}, we refer to the matrix
$\alpha_L$ as the lift of $\alpha$ in $\mathrm{S}(n, R \oplus I,0\oplus I)$.
For any $\varepsilon \in \E(n,R,I)$, one can have a nice representation of its lift $\varepsilon_L$, which we describe as follows: suppose that
$\varepsilon = \prod \eta_{ij}(x)\,\eta_{kl}(a)\,\eta_{ij}(-x)$,
where $a \in I$ and $x \in R$. Then one can check that $\varepsilon_L=
\prod \eta_{ij}((x,0))\,\eta_{kl}((0,a))\,\eta_{ij}((-x,0))$.
It is then clear that ${\varepsilon}_L \in \E(n, R \oplus I,0\oplus I)$. The
matrix ${\varepsilon}_L$ is called the lift of $\varepsilon$ in
$\E(n, R \oplus I)$.


\bl
\label{excision under localisation}
Let $R$ be a ring and $I$ be an ideal of $R$. Let $f\in R$ be a non-zero divisor in $R$. Then $(f,0)$ is a non-zero divisor in $R\oplus I$, and $(R\oplus I)_{(f,0)}$ is canonically isomorphic to $R_f\oplus I_f$. Moreover, if $\alpha\in \mathrm{S}(n, R,I)$ and $\varepsilon\in \E(n, R,I)$, then under the above isomorphism ${(\alpha_L)}_{(f,0)}\mapsto({\alpha_f})_L$ and ${(\varepsilon_L)}_{(f,0)}\mapsto ({\varepsilon_f})_L$.
\el	

\proof First, we prove that $(f,0)$ is a non-zero divisor in $R\oplus I$. For this, we choose an $(r,i) \in R \oplus I$ such that $(f,0)(r,i) = (0,0)$. Then
$(fr,fi) = (0,0)$, which implies that $r = 0$ and $i = 0$, as $f$ is a
non-zero divisor of $R$. This proves that $(f,0)$ is a non-zero divisor in $R \oplus I$.

Now we consider the canonical ring homomorphism
$\phi \colon (R \oplus I)_{(f,0)} \longrightarrow R_f \oplus I_f$,
defined by
$$\phi\!\left(\frac{(r,i)}{(f,0)^n}\right)
= \left(\frac{r}{f^n}, \frac{i}{f^n}\right).$$
Since $(f,0)$ is a non-zero divisor in $R\oplus I$, the map $\phi$ is an isomorphism.

Let $\widetilde{\phi}:\M_n((R \oplus I)_{(f,0)})\longrightarrow \M_n(R_f \oplus I_f)$ be the group homomorphism induce by $\phi$. We choose an $\alpha\in \mathrm{S}(n,R,I)$. Then there exists a matrix $\beta \in \M_{n}(I)$ such that $\alpha={\rm I}_{n}+\beta$, and thus $\alpha_f= \I_{n}+\beta_f$. Then one may observe that $\widetilde{\phi}((\beta_L)_{(f,0)})= ({\beta_f})_L$. Hence we have the following.
\begin{align*}
	\widetilde{\phi}((\alpha_L)_{(f,0)})&=\widetilde{\phi}((\I_{n}+{\beta_L})_{(f,0)} )\\
	&=\I_{n}+{(\beta_f)}_{L}\\
	&=(\I_{n}+{\beta_f})_{L}\\
    &={(\alpha_f)_L}
\end{align*}

Now for an $\varepsilon\in \E(n,R,I)$, one can write $\varepsilon= \prod \eta_{ij}(x)\eta_{kl}((a) \eta_{ij}(-x)$, where $x\in R$ and $a\in I$. Then we have the following.
	\begin{align*}
    \widetilde{\phi}(({\varepsilon_L})_{(f,0)})&= \widetilde{\phi}((\prod \eta_{ij}({x},{0})\eta_{kl}({0},{a}) \eta_{ij}({-x},{0}))_{(f,0)})\\
   &=\prod \eta_{ij}(\frac{x}{1},\frac{0}{1})\eta_{kl}(\frac{0}{1},\frac{a}{1}) \eta_{ij}(\frac{-x}{1},\frac{0}{1})\\&=(\varepsilon_f)_L
	\end{align*}
This concludes the proof.
\qed

\bl
		\label{analytic under excision}
		Let $R'\subset R$ be a subring of $R$. Let $I'\subset R'$ and $ I\subset R$ be two ideals such that 
		\begin{compactenum}[\quad\quad\quad(i)]
			\item $I'=\langle a_1,\dots, a_n\rangle $ is a prime ideal of $R'$, and
			\item $I=\langle a_1,\dots, a_n\rangle R$.
		\end{compactenum}
		Suppose that, there exists $h\in R'\setminus I'$ such that $R'\subset R$ is an analytic isomorphism along $h$. Then the natural inclusion map $R'\oplus I'\inj R\oplus I$ is an analytic isomorphism along $(h,0)$.
	\el
\proof Since $h$ is a non-zero divisor in $R$, it follows from the proof of Lemma \ref{excision under localisation} that $(h,0)$ is a non-zero divisor in $R\oplus I$ and hence in the subring $R'\oplus I'$.
%
	
	 Let `bar' denote going modulo $\langle (h,0)\rangle R\oplus I$. There is a natural homomorphism $\phi: \frac{R'\oplus I'}{\langle(h,0)\rangle R'\oplus I'}\to \frac{R\oplus I}{\langle (h,0)\rangle R\oplus I}$. Suppose that $\phi(\overline{(r',i')})=0$, then there exist  $r\in R$ and $i\in I$ such that $r'=hr$ and $i'=hi$. Since $R'\subset R$ is analytic isomorphism along $h$ we have $R'\cap hR=hR'$. Hence there exist $r'',i''\in R'$ such that $r'=hr''$ and $i'=hi''$. As $I'$ is a prime ideal, $hi''=i'\in I'$ and $h\notin I'$ imply that $i''\in I'$. Hence $(r',i')=(h,0)(r'',i'')\in \langle (h,0)\rangle (R'\oplus I')$. This proves that $\phi$ is injective.
	
	To prove the surjectivity of $\phi$, we consider an element $\overline{(r,i)}\in \frac{R\oplus I}{\langle (h,0)\rangle R\oplus I}$. Since $R'\subset R$ is analytic isomorphism along $h$ we have $R=R'+hR$. Hence there exist $r'\in R'$ and $s\in R$ such that $r=r'+hs$. Let $i=a_1x_1+a_2x_2+\dots + a_nx_n$ for some $x_p\in R$. As each $x_p$ can be written as $y_p+ hz_p$ for some $y_p\in R'$ and $z_p\in R$, we obtain that $i=i'+hj$, where $i'=a_1y_1+a_2y_2+\dots+ a_ny_n\in I'$ and $j=a_1z_1+a_2z_2+\dots+a_nz_n\in I$. This in particular shows that  ${\phi((r',i')+\langle (h,0)\rangle R'\oplus I'})=\overline{(r,i)}$. This completes the proof. \qed


\subsection{A relative version of Vorst's lemma}
    
	We begin this subsection with a lemma due to Vorst \cite[Lemma 2.4]{Vorst}. The following version appears in \cite[Lemma 3.12]{br}, where it is proved for $n \ge 6$ in the symplectic case. However, we remark that the same proof works in the case $n=4$ as in this case one can also prove \cite[Proposition 3.10]{br} using \cite[\S 3]{kopeiko}. We skip the proof to avoid repeating the same argument.
	\bl
	
	\label{vorst}
	Let $B \subset A$ be a subring, and let $h\in B$ be such that $B\hookrightarrow A$ is an analytic isomorphism along $h$. Let $n\geq 3$ in the linear case and $n\ge 4$ in the symplectic case. Then we have the following:
    	\begin{compactenum}
    \item If $\alpha \in \E({n},A_h)$, then there exist $\gamma  \in \E({n},A)$ and $\beta \in \E({n},B_h)$ such that $\alpha =\gamma_h \beta$. 
    \item If $\alpha\in \mathrm{S}({n},A)$ with $\alpha_h\in \E({n},A_h)$, then there exist  $\gamma \in \E({n},A)$  and  $\beta \in \mathrm{S}({n},B)$  such that (i) $\beta_h\in \E({n},B_h)$ and (ii) $\alpha=\gamma \beta$.
    	\end{compactenum}
	\el
	
	The following result is a relative version of the previous lemma.

	\bl
	\label{relative symplectic Vorst} 
	Let $R'\subset R$ be a subring of $R$. Let $I'\subset R'$ and $I\subset R$ be ideals of $R'$ and $ R$ respectively. Let $n\geq 3$ in the linear case and $n\ge 4$ in the symplectic case. Suppose that $R'\oplus I'\subset R\oplus I$ is analytic isomorphism along $(h,0)$ for some $h\in R'\setminus I'$. Then we have the following:
	\begin{compactenum}
		\item If $\alpha \in \E({n},R_h, I_h)$, then there exist $\gamma\in \E({n},R, I)$ and $\beta\in \E({n},R'_h, I'_h)$ such that $\alpha= \gamma_h \beta$.
		
		\item If $\alpha\in \mathrm{S}({n},R,I)$ with $\alpha_h\in \E({n},R_h,I_h)$,  then there exist $\beta\in \mathrm{S}(n,R',I')$ and $\gamma \in \E({n},R,I)$ such that $\alpha=\gamma\beta$ and $\beta_h\in \mathrm{E}(n,R'_h,I'_h)$.
	\end{compactenum} 
	\el
	
	\proof 
    
    Let us choose an $\alpha\in \E({n},R_h, I_h)$. Then the lift ${\alpha}_L$ of $\alpha$, as defined in Remark \ref{lifts under excision}, is in $ \E({n},R_h\oplus I_h,0\oplus I_h)$. Since $R'\oplus I'\hookrightarrow R\oplus I$ is an analytic isomorphism along $(h,0)$, using Lemma \ref{vorst}, there exist ${\Gamma}\in \E({n},R\oplus I)$ and $B\in \E({n},(R'\oplus I')_{(h,0)})$ such that\begin{equation}\label{VL}{\alpha}_L= {\Gamma}_{(h,0)}{B}.\end{equation}Moreover, by applying Lemma \ref{excision under localisation}, we may further conclude that $B \in \E(n, R'_h \oplus I'_h)$. Here one may note that we retain the same notation for $B$. In fact, throughout the remainder of the proof, we use the same notation for elements of $\mathrm{S}(n,(R' \oplus I')_{(h,0)})$ and $\mathrm{S}(n, R'_h \oplus I'_h)$.
    
{\textbf{Claim --} We claim that without loss of generality one may further assume that ${\Gamma}\in \E(n, R\oplus I, 0\oplus I)$ and $B\in \E(n, R'_h\oplus I'_h,0\oplus I'_h)$.

To establish this we denote `bar' by going modulo $0\oplus I$. Then we get that $\I_{n}= \overline{{\Gamma}_{(h,0)}} \overline{B}$ and $\overline{{\Gamma}} \in \E(n,R)$. Via the canonical inclusion $R\inj R\oplus I$ one can identify $\E(n,R)$ as a subgroup of $ \E(n, R\oplus I)$. Hence we can view $\overline{{\Gamma}}$ as an element of $\E(n,R\oplus I)$. Therefore, the matrix $${\Gamma} \overline{{\Gamma}}^{-1} \in \E(n, R\oplus I)\cap \mathrm{S}(n, R\oplus I, 0\oplus I).$$ Since the canonical projection map $R\oplus I\to R$ splits, applying Lemma \ref{K_2-lemma} we obtain that ${\Gamma} \overline{{\Gamma}}^{-1}\in  \E(n, R\oplus I, 0\oplus I)$. Thus localizing  we have ${\Gamma}_{(h,0)} \overline{{\Gamma}}_{(h,0)}^{-1}\in \E(n, R_h\oplus I_h, 0\oplus I_h)$.

In a similar way, working with $R'_h\oplus I'_h$, one can show that $ \overline{{B}}^{-1}{B}\in \E(n, R'_h\oplus I'_h, 0\oplus I'_h)$. Then using \ref{VL} we have the following identity.
		\begin{align*}
			{\alpha}_L &={\Gamma}_{(h,0)} {B}\\
			&= {\Gamma}_{(h,0)} \I_{n} B \\
			&= {\Gamma}_{(h,0)} {{\overline\Gamma}^{-1}_{(h,0)}} \overline{B}^{-1} B.
		\end{align*} 
Thus, the claim follows by replacing $\Gamma$ with $\Gamma \overline{\Gamma}^{-1}$ and $B$ with $\overline{B}^{-1} B$.
		
        One may note that the projection map $R_h\oplus I_h \to R_h$ defined by $(x,i)\mapsto x+i$ induces the maps $\E(n, R\oplus I, 0\oplus I)\to \E(n,R,I)$ and $ \E(n, R'_h\oplus I'_h,0\oplus I'_h)\to \E(n,R'_h, I'_h)$. Let $\gamma\in \E(n,R,I)$ and $\beta\in \E(n,R'_h, I'_h)$ be the images of $\Gamma$ and $B$ respectively. Then as a result we have $\alpha =\gamma_h \beta$. This proves (1).

	To prove (2) we consider the lift ${\alpha}_L\in\mathrm{S}(n,R\oplus I)$ of $\alpha$. Since $\alpha_h\in \E({n},R_h,I_h)$, by Lemma \ref{excision under localisation} we have ${(\alpha_L)}_{(h,0)}\in \E(n, (R\oplus I)_{(h,0)})$. As $R'\oplus I'\hookrightarrow R\oplus I$ is an analytic isomorphism along $(h,0)$, using Lemma \ref{vorst}, one can find $B\in \mathrm{S}(n, R'\oplus I')$ and $\Gamma\in \E(n, R\oplus I)$ such that $\alpha_L=\Gamma B$ and $B_{(h,0)}\in \mathrm{E}(n, R'_h\oplus I'_h)$. Further, up to a suitable substitution as incorporated earlier, we may assume that $B\in \mathrm{S}(n, R'\oplus I',0\oplus I')$ and $\Gamma\in \E(n, R\oplus I, 0\oplus I)$.

Let  $\gamma\in \E(n, R,I)$ be the image of ${\Gamma}$ and $\beta\in \mathrm{S}(n, R',I')$ be the image of $B$ under the map induced by the projection $R\oplus I\to R$ sending $(r,i)\mapsto r+i$. Therefore, we have $\alpha= \gamma \beta$. Moreover, using Lemma \ref{K_2-lemma}, we  get that $B_{(h,0)}\in \mathrm{E}(n,R'_h\oplus I'_h,0\oplus I'_h)$. Implying that $\beta_h\in \E(n, R'_h,I'_h)$. This completes the proof.\qed

	We conclude this subsection with a relative version of Suslin's lemma on patching of elementary matrices \cite[Lemma 3.7]{Suslin77}.

	\bc\label{patching} Let $R$ be a ring and $I\subset R$ be an ideal. Let $s,t\in R$ be two non-zero divisors such that $\langle s\rangle +\langle t\rangle=R$, and $st\notin I$. Further assume that $n\geq 3$ in the linear case and $n\ge 4$ in the symplectic case. Then for any $\alpha\in \E(n,R_{st},I_{st})$ there exist $\alpha_1\in \E(n,R_s,I_s)$ and $\alpha_2 \in \E(n, R_t, I_t)$  such that $\alpha= (\alpha_2)_s (\alpha_1)_t$.
	\ec
	
	\proof   By hypothesis, the localization map $R\to R_t$ is injective. Since $s$ and $t$ are co-maximal elements in $R$, we can write $su+tv=1$ for some $u,v\in R$. This yields the relation $(s,0) (u,0)+ (t,0) (v,0)= (1,0)$ in $R\oplus I$. Moreover, by Lemma \ref{excision under localisation}, we have $(s,0)$ and $(t,0)$ are non-zero divisors of $R\oplus I$. Therefore by Example \ref{example of a.i.}, the localization map $R\oplus I\to (R\oplus I)_{(t,0)}$ is an analytic isomorphism along $(s,0)$. Once again by Lemma \ref{excision under localisation}, we may assume that $R\oplus I\subset R_t\oplus I_t$ is analytic isomorphism along $(s,0)$. Since $\alpha\in \E(n, (R_t)_s, (I_t)_s)$, applying Lemma \ref{relative symplectic Vorst}, there exist $\gamma\in \E(n, R_t, I_t)$ and $\beta \in \E(n, R_s, I_s)$ such that $\alpha= \gamma_s \beta$. Now setting $\alpha_2= \gamma$ and $\alpha_1= \beta$ we get the desired result.\qed
		

\subsection{Completion of relative unimodular rows}
In this subsection, we study the completion of relative unimodular rows with respect to relative linear matrices. We begin this part with the following result, which is implicit in the proof of Lemma \cite[Theorem 5.2]{AG1}. However, for the sake of completeness we give the proof.

\bl
\label{principal implies general}
Let $R$ be a ring. Let $I\subset R$ be an ideal and $n\in \mathbb N$ be an integer. If $\Um_n(R, \ld a \rd)= e_1 \mathrm{S}(n, R, \ld a\rd)$ holds for all $ a\in I$, then $\Um_n(R,I)= e_1 \mathrm{S}(n, R,I)$.\el

\proof Since $e_1 \mathrm{S}(n,R,I) \subset \Um_n(R,I)$, it is enough to prove the reverse inequality. Let $v=(1-a_1, a_2, \dots, a_n)\in \Um_n(R,I)$. Then one may note that $$v  \eta_{13}(-a_3) \dots \eta_{1n}(-a_n)\eta_{12}(-a_2')= v',$$ where $\eta_{i,j}(\lambda)$'s are as defined in \ref{uniform def.}; $a_2'=a_2$ in the linear case, and $a_2'=a_2+a_3a_4+a_5a_6+\dots + a_{n-1}a_{n}$ in the symplectic case; and $v'=(1-a_1, a_1a_2, \dots, a_1 a_n)\in \Um_n(R, \ld a_1 \rd)$. Then, by hypothesis, we have $v'\in e_1 \mathrm{S}(n, R, \ld a_1 \rd)$. Since $a_i\in I$ for all $i$, one may note that $\mathrm{S}(n, R, \ld a_1 \rd) \subset \mathrm{S}(n, R,I)$ and $\eta:=\prod_{i=3}^n\eta_{1i}(-a_i) \eta_{12}(-a'_2)\in \E(n, R, I)\subset \mathrm{S}(n,R,I)$. Implying that $v\in e_1 \mathrm{S}(n,R, I)$. This completes the proof.\qed


The following result is a relative version of \cite[Theorem 2.4]{Suslin1984}, see also \cite[Proposition 3.1]{rvdk}. This result is essentially contained in \cite[Theorem 4.1]{AG1}, under the additional assumption that the base field is a perfect $C_1$ field. However, the same proof works in this case as well. For the sake of completeness we give the proof.

\bl
\label{Rao-vdK}
Let $R$ be an affine algebra of dimension $d$ over a field $k$ satisfying: For every  prime $p\leq d$ one of the following conditions is satisfied:
\begin{compactenum}
	\item $p\neq \rm{char}~(k)$, $\rm{c.d._p} (k)\leq 1$.
	\item $p=\rm{char}~(k)$ and $k$ is perfect.
\end{compactenum}
Let $I$ be an ideal of $R$. Then $\Um_{d+1}(R, I)= e_1 \SL_{d+1}(R,I)$.
\el

\proof Let us choose a relative unimodular row $v\in \Um_{d+1}(R,I)$. We consider the lift ${v}_L \in \Um_{d+1}(R\oplus I)$ of $v$. By Lemma \ref{keshari}, we obtain that $R\oplus I$ is an affine algebra of dimension $d$ over $k$. Therefore, applying \cite[Proposition 3.1]{rvdk} it follows that ${v}_L\in e_1 \SL_{d+1}(R\oplus I)$, i.e., we have ${v}_L {\Gamma}= e_1$ for some ${\Gamma}\in \SL_{d+1}(R\oplus I)$. Let `bar' denote going modulo $0\oplus I$. Going modulo $0\oplus I$ we obtain that $e_1 \overline{{\Gamma}}= e_1$. Without loss of generality, replacing ${\Gamma}$ by ${\Gamma} \overline{{\Gamma}}^{-1}$, we may assume that ${\Gamma}\in \SL_{d+1}(R\oplus I, 0\oplus I)$. Let $\gamma \in \SL_{d+1}(R,I)$ be the image of $\Gamma$ under the map induced by the projection $R\oplus I \to R$ sending $(r,i)\mapsto r+i$. Then we have $v\gamma =e_1$.  This completes the proof. \qed 

We recall the following result from \cite[Theorem 2.8]{Banerjee2025}.

\bl
\label{completion over real varieties}
Let $R$ be an affine algebra of dimension $d$ over $\mathbb{R}$ satisfying one of the following conditions, denoted by $\textbf{P}$:
\begin{compactenum}
	\item It has no real maximal ideal
	\item the intersection of all real maximal ideals has height at least $1$.
\end{compactenum}
Then $\Um_{d+1}(R)= e_1 \SL_{d+1}(R)$.
\el

\rmk Throughout this article, we denote the condition in Lemma~\ref{completion over real varieties} by $\textbf{P}$.

\bl
\label{Polynomial of P is P}
Let $R$ be a real affine algebra of dimension $d$ satisfying the condition $\textbf{P}$. Then $R[X]$ will also satisfy the condition $\textbf{P}$.
\el
\proof Suppose that $R$ has no real maximal ideal. With contrary let us assume that $\mathfrak{m}$ is a real maximal ideal of $R[X]$. Then we have a natural inclusion $\mathbb{R} \hookrightarrow R/ \mathfrak{m} \cap R \hookrightarrow R[X]/\mathfrak{m}$. Since $\mathfrak{m}$ is a real maximal ideal of $R[X]$, one must have $R/\mathfrak{m}\cap R \cong \mathbb{R}$. However, this contradicts the fact that $R$ has no real maximal ideal. This proves that $R[X]$ has no real maximal ideal whenever $R$ has no real maximal ideal.

Now, suppose we assume that the intersection of all real maximal ideals of $R$ is of height at least $1$. Let $J$ be the intersection of all real maximal ideals of $R[X]$. Now if $\mathfrak{n}$ is a real maximal ideal of $R[X]$, then by the argument as given in the previous paragraph it follows that $\mathfrak{n}\cap R$ is a real maximal ideal in $R$. Thus by the hypothesis on the height of the intersection of real maximal ideal in $R$ it follows that $\operatorname{ht}(J\cap R)\geq 1$. Hence there exists a non-zero divisor $a\in J\cap R$. However, then $a$ is also a non-zero divisor in $R[X]$, implying that $\operatorname{ht}(J)\ge 1$. This concludes the proof.    \qed

It follows from \cite[Lemma 2.13]{Banerjee2025} that if $R$ is an affine $\mathbb R$-algebra of dimension $d$ satisfying condition $\textbf{P}$, then so is $R\oplus I$ for any ideal $I$ of $R$. One can apply the same argument as in Lemma~\ref{Rao-vdK} to obtain the following result. Hence  we omit the proof to avoid repetition.

\bp\label{completion over real affine relative}
Let $R$ be a real affine algebra of dimension $d$ satisfying the condition $\textbf{P}$, and let $I$ be an ideal of $R$. Then $\Um_{d+1}(R,I)= e_1 \SL_{d+1}(R,I)$.
\ep


The next result is implicit in \cite[Corollary 4.3]{Banerjee2025}.

\bl
\label{coordinte power over real affine}
Let $R$ be a real affine algebra of dimension $d$ satisfying the condition \textbf{P}. Then for any 
\[
a=(a_1,\dots, a_{d+1})\in \Um_{d+1}(R)
\]
and any natural number $n$, there exists $\varepsilon\in \E_{d+1}(R)$ such that 
\[
a\varepsilon= (b_1,\dots, b_d, b_{d+1}^n).
\]
\el

 
 

 We note the following result which is implicit in the proof of \cite[Theorem 3.1]{rvdk}.
 
 \bl
 \label{suslin transformation of l^th power}
 
Let $n\in \mathbb{N}$, and let $n_p$ be the largest prime that divides $n$. Let $R$ be an affine algebra of dimension $d\ge 2$ over an infinite field $k$ such that for any  prime $p\leq n_p$ one of the following conditions is satisfied:
\begin{compactenum}
	\item $p\neq \rm{char}~(k)$, $\rm{c.d._p} (k)\leq 1$.
	\item $p=\rm{char}~(k)$ and $k$ is perfect.
\end{compactenum}
Then any unimodular row 
 \[
 a = (a_1, \ldots, a_{d+1}) \in \Um_{d+1}(R)
 \]can be transformed via elementary matrices 
 to a unimodular row of the form 
 \[
 b = (b_1, \ldots, b_{d+1}^n).
 \]
 \el
 
The next proposition is a relative version of the above result. The proof essentially follows the argument given in \cite[Theorem 4.4]{AG1}. For the sake of completeness, we give a sketch.
 
 \bp
 \label{relative coordinate power}
Let $n\in \mathbb{N}$, and let $n_p$ be the largest prime that divides $n$. Let $R$ be an affine algebra of dimension $d\ge 2$ over an infinite field $k$ such that for every  prime $p\leq n_p$ one of the following conditions is satisfied:
\begin{compactenum}
	\item $p\neq \rm{char}~(k)$, $\rm{c.d._p} (k)\leq 1$.
	\item $p=\rm{char}~(k)$ and $k$ is perfect.
\end{compactenum}
 Then any unimodular row 
 \[
 a = (a_1, \ldots, a_{d+1}) \in \Um_{d+1}(R, I),
 \]
there exists an $\varepsilon\in \E_{d+1}(R,I)$ such that $a\varepsilon=(b_1,\dots, b_d, b_{d+1}^n)$.
 \ep
 \proof Let $a=(a_1,a_2,\dots, a_{d+1})\in \Um_{d+1}(R,I)$ and $a_1= 1-\mu$, where $\mu\in I$. From Lemma \ref{principal implies general} one may observe the following.  \begin{equation*}
      v\equiv (1-\mu, \mu a_2, \dots, \mu a_{d+1})\pmod{\E_{d+1}(R,I)}
 \end{equation*}
 Hence it is enough to assume that $I=\langle \mu \rangle$ is a principal ideal. Applying Swan's version of Bertini theorem \cite[Theorem 1.3]{rgswan} we can add suitable multiple of $\mu a_2, \mu a_3, \dots, \mu a_{d+1}$ to $a_1$ and assume that $\frac{R}{\langle a_1\rangle}$ is a regular affine algebra of dimension $d-1$. Let `bar' denote going modulo $\langle a_1\rangle$. Then by Lemma \ref{suslin transformation of l^th power} and using the fact that $a_1$ and $\mu$ are comaximal, one can find $b_2, b_3, \dots, b_{d+1}\in I$ such that $$(\ol{b}_2, \ol{b}_3, \dots, \ol{b}_{d+1})\in \Um_d\big(\frac{R}{\langle a_1 \rangle}\big)\text{ and } (\ol{a}_2, \ol{a}_3, \dots, \ol{a}_{d+1})\ol{\varepsilon}=(\ol{b}_2, \ol{b}_3, \dots, \ol{b}_{d+1}^n),$$ where $\ol{\varepsilon}\in \E_{d-1}(\frac{R}{\langle a_1\rangle R})$. Again using the fact that $a_1$ and $\mu$ are comaximal, we can lift $\ol\varepsilon$ to a matrix $\varepsilon\in \E_d(I)$. Then we have $$(a_2, a_3,\dots, a_{d+1})\varepsilon\equiv (b_2, b_3, \dots, b_{d+1}^n)\pmod{I\langle a_1\rangle}.$$ Hence we obtain that$$ 	(a_1,a_2, \dots, a_{d+1})\varepsilon_1=(a_1, b_2,\dots, b_{d+1}^n),$$
 for some $\varepsilon_1\in \E_{d+1}(R,I)$. This completes the proof.\qed

We conclude this section with the following results that can be obtained using \cite[Corollary 4.3]{Banerjee2025}, \cite[Theorem 7.5]{FRS} (see also \cite[Theorem 4.4]{AG1}) and the argument given in Proposition~\ref{relative coordinate power}. Hence, we skip the proofs.

 \bp
 \label{relative coordinate power over real}
 Let $R$ be a real affine algebra of dimension $d\ge 2$ satisfying the condition \textbf{P} and $I$ be an ideal of $R$. Then for any unimodular row
  \[
 a=(a_1, \dots, a_{d+1})\in\Um_{d+1}(R,I)
 \]
 and any $n\in \mathbb{N}$, there exists an $\varepsilon\in \E_{d+1}(R,I)$ such that $a\varepsilon=(b_1,\dots, b_d, b_{d+1}^n)$.
 \ep
	
	\bp\label{relative coordinate power over alg closed}
	Let $R$ be a regular affine algebra of dimension $d\geq 4$ over an algebraically closed field $k$. Let $I$ be an ideal of $R$. 
    Then for any unimodular row
  \[
 a=(a_1, \dots, a_{d})\in\Um_{d}(R,I)
 \]
 and for any $n\in \mathbb{N}$ which is prime to $\rm{char}(k)$, there exists an $\varepsilon\in \E_{d}(R,I)$ such that $a\varepsilon=(b_1,\dots, b_{d-1}, b_{d}^n)$. In particular, if $\frac{1}{(d-1)!}\in k$, then $\Um_d(R,I)=e_1 \SL_d(R,I)$. 
	\ep

	\section{A symplectic $\mathrm{K_2}$-Lemma}\label{K2}
	
The main result of this section is Theorem \ref{symplectic excision}, whose linear version is due to Suslin \cite[$\S 4$]{Suslin77}. To establish this, we require some preparations. We begin with the following remark.
	
	\smallskip
	
	\remark
		
	Let $R$ be a ring. Let $n \geq 3$ and let $R^{2n}$ denote the free $R$-module with the canonical basis $\{e_{1}, e_{2}, \dots, e_{2n}\}$. Any element $v \in R^{2n}$ can be written as $v=\sum_{i=1}^{2n} v_i e_i$. We equip $R^{2n}$ with the structure of a symplectic module with respect to the alternating bilinear form $\langle \,,\, \rangle$ defined by $\langle e_i,e_j\rangle = (-1)^{i+1} \delta_{i,\sigma(j)}$. We shall use the notation $\varepsilon_i$ to denote $(-1)^{i+1}$. The symplectic group $\Sp_{2n}(R)$ coincides with the group of automorphisms of $R^{2n}$ which preserves this alternating form, see \cite{lavrenov} for details.
	
	We recall the following definition from \cite{lavrenov}.
	\bd(Eichler–Siegel–Dickson transformations)
		Let $a\in R$ and let $u,v\in R^{2n}$ be such that $\langle u,v\rangle=0$. The transformation $T(u,v,a):R^{2n}\to R^{2n} $
		defined by $$T(u,v,a)(w)=w+(\langle v,w\rangle +a\langle u,w \rangle)u+\langle u,w \rangle v$$
		is called symplectic ESD-transformation.
	\ed
	
We recall the following results from \cite[Lemmas 1 and 2]{lavrenov}.
	
	\bl
		\label{ESD relation I}
		Let $u,v,w\in R^{2n}$ be three vectors such that $\langle u,v \rangle =0= \langle u,w \rangle$, and $a,b\in R$. Then 
		\begin{compactenum}
			\item $T(u,v,a)\in \Sp(R^{2n})= \Sp_{2n}(R)$,
			\item $T(u,v,a)T(u,w,b)=T(u,v+w, a+b+\langle v,w \rangle)$,
			\item $T(u,av,0)=T(v,au,0)$,
			\item $gT(u,v,a)g^{-1}=T(gu,gv,a)$ for all $g\in \Sp(R^{2n})=\Sp_{2n}(R)$.
		\end{compactenum}
	\el
	
	\bl
		\label{ESD relation II}
		Let $u,v\in R^{2n}$ be such that $u_{i}=u_{\sigma(i)}=v_i=v_{\sigma(i)}=0, \langle u,v \rangle =0$, and let $ a\in R$. Then we have the following commutator formula.
	$$[T(e_i,u,0), T(e_{\sigma(i)},v,a)]=T(u,\epsilon_iv,a) T(e_{\sigma(i)}, -\varepsilon_{\sigma(i)} a u,0)$$
	\el
	
	\bd

		Let $n \geq 3$ and let $i,j \in \{1,\dots,2n\}$ with $i \notin \{j,\sigma(j)\}$. For any $a \in R$, the transformation $T_{i,j}(a)$, defined by $T(e_i,\varepsilon_{\sigma(j)}a e_{\sigma(j)},0)$, and the transformation $T_{i,\sigma(i)}(a)$, defined by $T(e_i,0,\varepsilon_i a)$, are called elementary symplectic transvections. The subgroup of $\Sp_{2n}(R)$ generated by these elementary symplectic transvections is called the elementary symplectic group and is denoted by $\ESp(R^{2n})$.
	\ed
	
	The next lemma identifies the group $\ESp(R^{2n})$ with $\ESp_{2n}(R)$. To establish this, it suffices to verify that the matrix of an elementary symplectic transvection, with respect to a fixed basis of $R^{2n}$, is a generator of the elementary symplectic group $\ESp_{2n}(R)$. A direct computation shows that, with respect to the canonical basis $\{e_1, e_2, \dots, e_{2n}\}$, the matrix of $T_{i,j}(a)$ is $se_{ij}(a)$ for $i \neq \sigma(j)$, while the matrix of $T_{i,\sigma(i)}(a)$ is $se_{i\sigma(i)}(a)$. We therefore omit the proof.
	
	\bl
	\label{ESD generators are elementary}
		With respect to the above notation, one has $\ESp(R^{2n})\cong\ESp_{2n}(R)$.
	\el

	The followig lemma establishes some useful identities. The proof essentially uses Lemma \ref{ESD relation I}, and hence we skip the proof.
	
	\bl
The elementary symplectic group $\ESp_{2n}(R)$ possesses the following properties:
\begin{align}
	&se_{ij}(a)=se_{\sigma(j), \sigma(i)}(-\varepsilon_i \varepsilon_j a),\label{E0}\tag{E0}\\
	&se_{ij}(a)se_{ij}(b)=se_{ij}(a+b),\label{E1}\tag{E1}\\
	&[se_{ij}(a),se_{hk}(b)]=1, \text{for $h\notin \{j,\sigma(i)\}, k\notin \{i,\sigma(j)\}$}\label{E2}\tag{E2}\\
	&[se_{ij}(a),se_{jk}(b)]=se_{ik}(ab), \text{for $i\notin \{\sigma(j),\sigma(k)\}, j\neq \sigma(k),$}\label{E3}\tag{E3}\\
	&[se_{i,\sigma(i)}(a),se_{\sigma(i),j}(b)]=se_{ij}(ab) se_{\sigma(j),j}(\varepsilon_i \varepsilon_j ab^2),\label{E4}\tag{E4}\\
	&[se_{ij}(a),se_{j,\sigma(i)}(b)]=se_{i,\sigma(i)}(2ab).\label{E5}\tag{E5}
\end{align}
\el

We recall next the definition of the symplectic Steinberg group $\StSp_{2n}(R)$.
	
	\bd
		The symplectic Steinberg group $\StSp_{2n}(R)$ is defined as the group generated by formal symbols $X_{ij}(a)$, where $i \neq j$ and $a \in R$, subject to the following Steinberg relations:
		\begin{align}
			&X_{ij}(a)=X_{\sigma(j),\sigma(i)}(-\varepsilon_i \varepsilon_j a),\label{R0}\tag{R0}\\
			&X_{ij}(a)X_{ij}(b)=X_{ij}(a+b),\label{R1}\tag{R1}\\
			&[X_{ij}(a),X_{hk}(b)]=1, \text{for $h\notin \{j,\sigma(i)\}, k\notin \{i,\sigma(j)\}$}\label{R2}\tag{R2}\\
			&[X_{ij}(a),X_{jk}(b)]=X_{ik}(ab), \text{for $i\notin \{\sigma(j),\sigma(k)\}, j\neq \sigma(k),$}\label{R3}\tag{R3}\\
			&[X_{i,\sigma(i)}(a),X_{\sigma(i),j}(b)]=X_{ij}(ab) X_{\sigma(j),j}(\varepsilon_i \varepsilon_j ab^2),\label{R4}\tag{R4}\\
			&[X_{ij}(a),X_{j,\sigma(i)}(b)]=X_{i,\sigma(i)}(2ab).\label{R5}\tag{R5}
		\end{align}
	
	\ed
		It follows from \cite[Lemma 3]{lavrenov} that there is a natural epimorphism $$\phi: \StSp_{2n}(R)\twoheadrightarrow \ESp_{2n}(R)$$ sending the generators $X_{ij} (a)\mapsto T_{ij} (a)$. 
	
	\bd
		The kernel of the natural map $\phi: \StSp_{2n}(R)\twoheadrightarrow \ESp_{2n}(R)$ is denoted by $\mathrm{KSp_{2,2n}}(R)$.
	\ed
	
	
The map $\iota:\StSp_{2n}(R)\to \StSp_{2n+2}(R)$ sending $X_{ij}(r)\mapsto X_{ij}(r)$ is a homomorphism, and we define $\StSp(R) := \varinjlim \StSp_{2n}(R)$. Then there is a natural surjection $\phi:\StSp(R)\twoheadrightarrow \ESp(R)$, by abuse of notation we still call it $\phi$. The kernel of $\phi$ is denoted by $\mathrm{KSp}_2(R)$.

We now recall the following definition from \cite[5.5F*]{HO}.

	\bd
		For $r,s\in R^{\times}$, we define the following:

		\begin{compactenum}[\quad\quad\quad\quad]
		\item $sw_{ij}(r):=X_{ij}(r) X_{ji}(-r^{-1}) X_{ij}(r)$,
		\item $sh_{ij}(r):=sw_{ij}(r)sw_{ij}(-1)$,
		\item $\{r,s\}_{ij}:=sh_{ij}(rs)sh_{ij}(r)^{-1} sh_{ij}(s)^{-1}$.
		
	\end{compactenum}

	\ed
	
	We record the following result from \cite[5.6A*]{HO}.
	
	\bl\label{symbol}
		If $r,s\in R^{\times}$, then the symbol $\{r,s\}_{ij}\in \mathrm{KSp_{2,2n}}(R)$ and is independent of the choice of the pair $(i,j)$ such that $1\leq i\neq j\leq 2n$. 
	\el
	
	\remark
		There are two types of symbols $\{a,b\}_{ij}$ in $\mathrm{KSp_{2,2n}}(R)$, depending on whether $i = \sigma(j)$ or not. In both cases, the symbols are independent of the choice of the pair $(i, j)$ with $1 \leq i \neq j \leq 2n$. We denote the symbols $\{r,s\}_{13}$ by $\{r,s\}$ and $\{r,s\}_{12}$ by $[r,s]$. The relation between these two symbols is $\{r,s\}=[r^2,s]$ (cf. \cite[page no. 279, 5.6.2]{HO}).

The proof of the next lemma can be found in \cite[page no. 279, 5.6.3]{HO}.
	
	\bl
		Suppose that $R$ is a field and  $r,s,t\in R^{\times}$, then 
		\begin{compactenum}
			\item $[r,1]=[1,r]=1$,
			\item $[rs,t][r,s]=[r,st][s,t]$,
			\item $[r,s]=[s^{-1},r]$,
			\item $[r,s]=[r,-rs]$,
			\item $[r,s]=[r,(1-r)s]$, if $r\neq 1$
		\end{compactenum}
	\el
	
	The following lemma is due to Matsumoto see \cite{Matsumoto} or \cite[page no 377, 6.5.11]{HO}.
	\bl
		\label{symplectic Matsumoto}
		Let $n\geq 3$ and let $R$ be a field. Then the set $\{[r,s]: r,s\in R^{\times}\}$ generates $\mathrm{KSp_{2,2n}}(R)$, and with the relations
		\begin{compactenum}
			\item $[r,1]=[1,r]=1$,
			\item $[rs,t][r,s]=[r,st][s,t]$,
			\item $[r,s]=[s^{-1},r]$,
			\item $[r,s]=[r,-rs]$,
			\item $[r,s]=[r,(1-r)s]$, if $r\neq 1$
		\end{compactenum}
		provides a presentation of $\mathrm{KSp_{2,2n}}(R)$ as an abelian group.
	\el
	
	%
	%
	%
	%
	
	Various approaches to the relative Steinberg group are described in \cite{lavrenov}, \cite{keune}, \cite{loday}, and \cite{sinchuk}. We follow the setup of \cite{lavrenov}. 

We prove the following result for the sake of completeness.

\bp 
\label{kernel of steinberg group}
Let $f: R\twoheadrightarrow S$ be a surjective homomorphism with the kernel $I$. Then the kernel of the map $f_*:\StSp_{2n}(R)\to \StSp_{2n}(S)$ induced by $f$ is the normal subgroup of $\StSp_{2n}(R)$ generated by the set $\{X_{ij}(a): a\in I\}$. 
\ep

\proof Let $H$ be the normal subgroup of $\StSp_{2n}(R)$ generated by $\{X_{ij}(a):a\in I\}$. Clearly $f_*(H)=1$, and hence $H \subset \ker({f_*} )$. Therefore, the map $f_*$ induces a homomorphism $\ol{f_*}: \StSp_{2n}(R)/H \to \StSp_{2n}(S)$ making the following triangle commute.

$$\begin{tikzcd}
	\StSp_{2n}(R) \arrow{rr}{f_*}\arrow{rd}  && \StSp_{2n}(S)  \\
&\frac{\StSp_{2n}(R)}{H} \arrow{ur}[swap]{\ol{f_*}}  
\end{tikzcd}$$
To conclude the proof, it suffices to show that the map $\overline{f_*}$ is an isomorphism. The remainder of the proof is devoted to establishing this. 

We define a homomorphism
$g: \StSp_{2n}(S) \longrightarrow \StSp_{2n}(R)/H$
as follows: for a generator $X_{ij}(s) \in \StSp_{2n}(S)$, we choose an element $r \in R$ such that $f(r) = s$, and define $g(X_{ij}(s)) := X_{ij}(r)H.$

To verify that $g$ is well defined, suppose that we choose two different lifts $r, r' \in R$ of $ s$, i.e., satisfying
$f(r) = f(r') = s$. Then the element $a:=r-r'\in I$. Since $X_{ij}(a) \in H$, we have $$g(X_{ij}(r'))=g(X_{ij}(r)),$$ and hence, the definition of 
 $g(X_{ij}(s))$ is independent of the choice of a lift $r$ of $s$. Moreover, since the generators $X_{ij}(s)$ satisfy the relations \ref{R0} to \ref{R5}, its image $X_{ij}(r)H$ also satisfy the relations \ref{R0} to \ref{R5}. Therefore, the map $g:\StSp_{2n}(S) \to \StSp_{2n}(R)/H$ is a well defined. Now it follows from the definition of $g$ that both the compositions $g\circ\ol{f_*}$ and $\ol{f_*}\circ g$ are the identity maps. Hence $\ol{f_*}$ is an isomorphism. This concludes the proof.\qed

We now define the relative symplectic Steinberg group.
	
\bd
Let $R$ be a ring and let $I$ be an ideal of $R$. The relative symplectic Steinberg group $\StSp_{2n}(R,I)$ is defined as the kernel 
of the homomorphism $\StSp_{2n}(R) \to \StSp_{2n}(R/I)$, induced by the canonical projection $R \twoheadrightarrow R/I$.
\ed

	In the linear case, the following result is due to Suslin \cite{Suslin77}. Our proof follows the argument given in \cite[Lemma 4.3]{GM}.
	
\bp
		\label{symlpectic K_2 lemma}
		Let $R$ be a ring and $I\subset R$ be an ideal such that the canonical map $\psi:\mathrm{KSp_{2,2n}}(R)\to\mathrm{ KSp_{2,2n}}( R/I)$ is surjective. Then $\ESp_{2n}(R,I)=\Sp_{2n}(R,I)\cap \ESp_{2n}(R)$.
	\ep
	
\proof We consider the following diagram of short exact sequences.
	
	\[
	\begin{tikzcd}
		& &0\arrow{d} &0\arrow{d}\\
		& &\StSp_{2n}(R,I)\arrow{r}{\widetilde{\phi}} \arrow{d}{} & \ESp_{2n}(R)\cap \Sp_{2n}(R,I)\arrow{d}{}\\
		0\arrow{r} &\mathrm{KSp_{2,2n}}(R)\arrow{r} \arrow{d}{\psi}&\StSp_{2n}(R)\arrow{r}{\phi} \arrow{d} & \ESp_{2n}(R)\arrow{r} \arrow{d}{}&0\\
		0\arrow{r} &\mathrm{KSp_{2,2n}}(R/I)\arrow{r} &\StSp_{2n}(R/I)\arrow{r}{}&\ESp_{2n}(R/I) \arrow{r} &0
	\end{tikzcd}
	\]
	Here $\widetilde{\phi}$ is induced by the map $\phi:\StSp_{2n}(R)\to \ESp_{2n}(R)$. Using Proposition \ref{kernel of steinberg group}, the group $\StSp_{2n}(R,I)$ is the normal subgroup of $\StSp_{2n}(R)$ generated by Steinberg symbols $X_{ij}(a)$, where $a\in I$. Hence we have $\widetilde{\phi}(\StSp_{2n}(R,I))=\ESp_{2n}(R,I)$. Therefore, it is enough to show that $\widetilde{\phi}$ is surjective. However, applying Snake lemma and the fact that $\psi$ is surjective, this follows.\qed
	
	Now we are ready to prove a Symplectic $\mathrm{K_2}$-Lemma. The linear version of this theorem is due to Suslin \cite[$\S$ 4]{Suslin77}.

	\bt
		\label{symplectic excision}
		Let $(R,\mathfrak{m})$ be a local ring. Then for all $n\geq 3$, we have the following. $$\Sp_{2n}(R[X],\mathfrak{m}[X])\cap \ESp_{2n}(R[X])=\ESp_{2n}(R[X],\mathfrak{m}[X])$$ 
	\et
	
	\proof We first observe that, in view of Proposition \ref{symlpectic K_2 lemma}, it is enough to show that the canonical map $\psi:\mathrm{ KSp_{2,2n}}(R[X]) \to\mathrm{ KSp_{2,2n}}(R[X]/\mathfrak{m}[X])$
	is surjective. The proof is devoted to establishing this fact.
	
	Let $k=R/\mathfrak{m}$ be the residue field of $(R,\mm)$. Since $k$ is a field, applying \cite[Korollar zu Satz 1]{Rehmann} we have $\mathrm{KSp_{2,2n}}(k[X])\simeq\mathrm{KSp_{2,2n}}(k)$. Hence by Lemma \ref{symplectic Matsumoto}, the group $\mathrm{KSp_{2,2n}}(k[X])$ is generated by symbols {$[u,v]$ such that $u,v\in (k[X])^{\times}= k\setminus 0$. Since $R$ is a local ring, the units of $k$ can be lifted 
	to the units of $R$, and hence to the units of $R[X]$. For each symbol $[u,v]$, we choose lifts $r,s\in R^{\times}\subset (R[X])^{\times}$ of $u$ and $v$. Then it follows from Lemma \ref{symbol} that the symbols $[r,s]\in\mathrm{ KSp_{2,2n}}(R[X])$. This concludes the proof.\qed

		\section{A relative version of Vorst's theorem}\label{Vorst}

The main purpose of this section is to prove Theorem \ref{relative vorst theorem}, which plays a crucial role in the context of this article. In the absolute case, this result was established by Vorst \cite[Theorem 3.3]{Vorst} for linear groups and by Basu–Rao \cite[Theorem 3.8]{br} for symplectic groups. To proceed, we first recall the following result from \cite[Theorem 3.9]{AG1}.

	\bl
		\label{equality of linear and symplectic}
		Let $R$ be a ring and $I$ be an ideal of $R$. Let $v\in \Um_{2n}(R,I)$. Then $v\E_{2n}(R,I)=v\ESp_{2n}(R,I)$, for all $n\geq 2$.
	\el

	%
%
%
%
%
%
The next result addresses the monic inversion principle for linear and symplectic elementary matrices. In the linear case the proof is given in \cite[Corollary 5.7]{Suslin77}, and in the symplectic case the proof is given in \cite[Theorem 3.10]{kopeiko}.

	\bl
		\label{monic inversion for symplecic matrix}
		Let $n$ be an integer. In addition, in the symplectic case, assume that $n \ge 4$. Let $R$ be a ring and $\alpha\in \mathrm{S}({n}, R[X])$.  If there exists a monic polynomial $f\in R[X]$ such that $\alpha_f\in \E({n},R[X]_f)$, then $\alpha\in \E({n},R[X])$.
	\el
	
	The following result is a relative version of the previous lemma.
	
	\bc
		\label{relative monic inversiion for symplectic matrix}
Let $n$ be an integer. In addition, in the symplectic case, assume that $n \ge 4$. Let $R$ be a ring and $I$ be an ideal of $R$. Let $\alpha\in \mathrm{S}({n},R[X],I[X])$. If there exists a monic polynomial $f\in R[X]$ such that $\alpha_f\in \E({n},R[X]_f,I[X]_f)$, then $\alpha\in \E(n,R[X],I[X])$.
	\ec
	
	\proof Consider the lift ${\alpha}_L$ of $\alpha$ as defined in \ref{general lift}. Using Lemma \ref{lift of invertible} and Corollary \ref{excision of relative symplectic} we obtain that the lift ${\alpha}_L\in\mathrm{S}({n},(R\oplus I)[X],(0\oplus I)[X])$. Then one may notice that $(f,0)$ is a monic polynomial in $(R\oplus I)[X]$. Since $\alpha_f\in \E({n},R[X]_f,I[X]_f)$ by Remark \ref{lifts under excision} we have $(\alpha_L)_{(f,0)}\in \E({n},(R\oplus I)[X]_{(f,0)})$. Hence using Lemma \ref{monic inversion for symplecic matrix}, we get that ${\alpha}_L\in \E({n},(R\oplus I)[X])$. Now applying Lemma \ref{K_2-lemma}, one may observe the following. $${\alpha}_L\in \mathrm{S}({n},(R\oplus I)[X],(0\oplus I)[X])\cap \E({n},(R\oplus I)[X])=\E({n},(R\oplus I)[X],(0\oplus I)[X])$$ Then we consider the projection map from $$R[X]\oplus I[X]\to R[X]\text{ sending }(f(X),g(X))\mapsto f(X)+g(X).$$ Under this map  $\E({n},(R\oplus I)[X],(0\oplus I)[X])$ goes to $\E({n},R[X],I[X])$ and ${\alpha_L}$ goes to $\alpha$. Hence, we obtain that $\alpha\in \E({n},R[X],I[X])$. This concludes the proof. \qed

We now present a relative version of Vorst's theorem \cite[Theorem 3.8]{br}.

	\bt
		\label{relative vorst theorem}
		Let $R$ be a regular $k$-spot of dimension $d$ over an infinite field $k$. Let $I$ be an ideal of $R$ such that $R/I$ is regular. Then $$\mathrm{S}({n},R[X],I[X])=\E({n},R[X],I[X]),$$ where in the linear case $n\ge 3$, and in the symplectic case $n\ge 6$.
	\et
	
	\proof We begin with some reduction steps. If $I=R$, then the result follows from Lemma \ref{Symplectic K_1}. Further, if $d=0$, then $R$ is a field and $I=0$. In this case, we have $\mathrm{S}({n},R[X],I[X])=\{I_{n}\}=\E({n},R[X],I[X])$. Hence, without loss of generality, we may assume that $I$ is a proper ideal of $R$ and $d>0$. Since $R/I$ is a regular local ring, the ideal $I$ is a prime ideal of $R$. Now, if $I$ is the maximal ideal then the result follows from \cite[$\S$ 4]{Suslin77} and Theorem \ref{symplectic excision}. Hence one may further assume that $I\subsetneq \mathfrak{m}$. Moreover, we note that the case of an arbitrary infinite ground field can be reduced to a separating infinite ground field by the same argument as given at the end of the proof of \cite[Theorem 3.13]{kcr}. Hence we further assume that $R$ is a regular $k$-spot with separating infinite ground field $k$. The remainder of the proof is divided into the following steps.
	
	%
	%
	
	\paragraph{\textbf{Step - 1}} It follows from Lemma \ref{quotient of regular is regular} that there exists a regular system of parameters $f_1,f_2,\dots, f_d\in R$ such that $$I=\langle f_1,f_2,\dots, f_c\rangle, \text{ for some }c< d .$$ Since $R$ is a regular local ring, using \cite[Proposition 2.2.5, page no. 66]{BH} we may assume that $\{f_1,\dots,f_d\}$ is a minimal set of generators of $\mathfrak{m}/\mathfrak{m}^2$ under a suitable substitution. Now, let $S=R\setminus I$ and let $\alpha\in \mathrm{S}({n},R[X],I[X])$. Then we have $\alpha_S\in \mathrm{S}({n},R_S[X],I_S[X])$. Applying Lemma \ref{Symplectic K_1}, we have $\alpha_S\in \E({n},R_S[X])$. Therefore, the matrix $\alpha_S\in \mathrm{S}({n},R_S[X], I_S[X])\cap \E({n},R_S[X])$. Since $I_S$ is the maximal ideal of $R_S$, by \cite[$\S$ 4]{Suslin77} and Theorem \ref{symplectic excision}, it follows that $\alpha_S\in \E({n},R_S[X], I_S[X])$. One can clear denominators
	and find a suitable element $g\in S$ such that $\alpha_g\in \E({n},R_g[X], I_g[X])$. As $f_1, f_2,\dots, f_{d-1}$ are part of a basis of $\mathfrak{m}/\mathfrak{m}^2$ and $R$ is a regular local ring, it follows from the proof of \cite[Theorem 2.8]{budh} that the sequence $g^2,f_1,f_2,\dots, f_{d-1}$ is a regular sequence in $R$. Hence by Theorem \ref{NRao}, there exists a field $K\supset k $ and a regular $K$-spot $R'\subset R$ such that 
	
	\begin{compactenum}[\quad\quad\quad (1)]
		\item  $R'=K[X_1,X_2,\dots, X_d]_{\ld X_1,\dots, X_{d-1}, \varphi(X_d)\rd}$, where $\varphi(X_d) \in K[X_d]$ is an irreducible monic polynomial, $X_i=f_i$ for $i=1,\dots,d-1$ and
		\item $R'\hookrightarrow R$ is an analytic isomorphism along $h$ for some $h\in g^2R\cap R'$.
	\end{compactenum}
	Applying \cite[Lemma 2.17]{kcr}, it follows that $g^2$ and $h$ differ by a unit in $R$, hence we may further assume that $h=g^2$. Now, if $h\not\in \mathfrak{m}$, then $\alpha_h\in \E({n},R_h[X], I_h[X])$ implies that $\alpha\in \E({n},R[X], I[X])$. Therefore, without loss of generality we may further assume that $h\in \mathfrak{m}$. Moreover, applying the same argument finitely many times we may assume that every factor of $h$ lies in $\mathfrak{m}$.
	
	We define $I':=\langle X_1,X_2,\dots, X_c\rangle R'\subset I$. Since $R'/I'$ is a localization of a polynomial extension over the field $K$, the ideal $I'$ is a prime ideal of $R'$. Hence applying Lemma \ref{analytic under excision}, we obtain that the canonical inclusion $$R'[X]\oplus I'[X]\hookrightarrow R[X]\oplus I[X]$$ is an analytic isomorphism along $(h,0)$. As $\alpha_h\in \E({n},R_h[X], I_h[X])$, it follows from a relative version of Vorst's lemma \ref{relative symplectic Vorst} that there exist $\beta\in \mathrm{S}({n},R'[X],I'[X])$ and $\gamma \in \E({n},R[X], I[X])$ such that 
	\begin{compactenum}[\quad\quad\quad (i)]
		\item $\alpha=\gamma \beta$ and
		\item  $\beta_h\in \E({n},R'_h[X], I'_h[X])$.
	\end{compactenum}
	One may note that to conclude the proof it is enough to show that $\beta\in \E({n},R'[X], I'[X])$. The remaining part is devoted to show this only. 
	
	\paragraph{\textbf{Step - 2}} Clearing denominators, if necessary, we may further assume that $$h\in \langle X_1,\dots, X_{d-1}, \varphi(X_d)\rangle\subset K[X_1,X_2,\dots, X_d]. $$Applying Nagata's change of variables theorem \ref{Nagata}, there exists a change of variables in $K[X_1,\dots,X_d]$, namely $X_i \mapsto X_i+\varphi(X_d)^{r_i}$ for $1 \leq i \leq d-1$ and $X_d \mapsto X_d$, such that $h$ becomes a monic polynomial in $X_d$ over $K[X_1,\dots,X_{d-1}]$.

\paragraph{\textbf{Claim}} We claim that $h$ is comaximal with any element belongs to $ K[X_1,X_2,\dots, X_d]\setminus \langle X_1,\dots,X_{d-1}, \varphi(X_d)\rangle $.

	\paragraph{Proof of the claim} Let $f\in K[X_1,X_2,\dots, X_d]\setminus \langle X_1,\dots,X_{d-1}, \varphi(X_d)\rangle $. Then we show that $h$ and $f$ are comaximal. Since an automorphism maps units to units, it suffices to show that, under some automorphism $h$ and $f$ become comaximal. 
	 
	  As every factor of $h$ lies in $\mm$ and $f\notin \mm$, the elements $f$ and $h$ does not have any common factor. Then by Lemma \ref{ojanguren}, there exists an automorphism of $K[X_1,\dots, X_d]$ under which $h$ and $f$ are monic in $X_d$ such that $h(0,\dots,0, X_d)$ and $f(0,\dots, 0, X_d)$ are coprime. Hence using Lemma \ref{coprime_imply_comaximal}, we get that $f$ and $h$ are comaximal. This proves the claim.

	\paragraph{\textbf{Step - 3}}We define $R'':=K[X_1,\dots, X_{d-1}]_{\langle X_1,\dots, X_{d-1}\rangle }[X_d]$. Then one may note that we still have $h\in \langle X_1,\dots,X_{d-1}, \varphi(X_d)\rangle \subset K[X_1,X_2,\dots, X_d]$.  Further, applying Lemma \ref{pseudo-weierstrass}, we get that the inclusion $R''\hookrightarrow R'$ is an analytic isomorphism along $h$. Hence the inclusion $R''[X]\hookrightarrow R'[X]$ is also an analytic isomorphism along $h$. We consider the prime ideal $I''=\langle X_1,\dots, X_c\rangle R''$. Then by Lemma \ref{analytic under excision}, we have $$R''[X]\oplus I''[X]\hookrightarrow R'[X]\oplus I'[X]$$is an analytic isomorphism along $(h,0)$. Since $\beta_h\in \E({n},R'_h[X], I'_h[X])$, by relative symplectic version of Vorst's lemma \ref{relative symplectic Vorst}, there exist $\delta\in \E({n},R'[X], I'[X])$ and $\theta \in\mathrm{S}({n},R''[X], I''[X])$ such that $\beta= \delta \theta$ and $\theta_h\in \E({n},R''_h[X], I''_h[X])$. Again as before, one may observe that to establish the fact that $\beta\in \E({n},R'[X], I'[X])$ it is enough to prove that $\theta\in \E({n},R''[X], I''[X])$. In the remaining part we show this only. 
	
Let us take $$A=(K[X_1,\dots, X_{d-1}]_{\langle X_1,\dots, X_{d-1}\rangle }[X]).$$ Then $A[X_d]= R''[X]$, and moreover, it follows from the definition of $I''$ that $$I''[X]=\langle X_1,\dots, X_c\rangle R''[X]=\langle X_1,\dots, X_c\rangle A[X_d]$$is an extended ideal of the polynomial ring $A[X_d]$, and $h\in A[X_d]$ is monic polynomial in $X_d$. Therefore, applying a relative version of a monic inversion principle as given in Corollary \ref{relative monic inversiion for symplectic matrix}, we obtain the following. $$\theta\in \E({n},R''[X], I''[X])$$This concludes the proof.\qed
	
	\rmk If Matsumoto's lemma \ref{symplectic Matsumoto} holds for $n=2$, then Theorem \ref{relative vorst theorem} also proves that
$	\mathrm{Sp}_4(R[X],I[X])=\mathrm{ESp}_4(R[X],I[X])$,
	where $R$ and $I$ are as in the theorem. However, we do not know whether Matsumoto's lemma holds for $n=2$.
	
We conclude this section recalling the following relative Local-Global Principle from \cite[Theorem 4.5]{acr}.

\bt \label{relative local global}

Let $n\ge 3$, and let $\alpha(X) \in \operatorname{S}(n, R[X], I [X])$, with $\alpha(0) = I_n$. If $\alpha_{\mathfrak{m}}(X) \in \E(n, R_\mathfrak{m}[X], I_\mathfrak{m}[X])$, for
every maximal ideal $\mathfrak{m}$ of $R$, then $\alpha(X) \in \E(n, R[X], I [X])$.
\et

	\bc
		Let $R$ be a regular affine $k$-algebra of dimension $d$, where $k$ is an infinite field. Let $I$ be an ideal of $R$ such that $R/I$ is also regular. Let $\alpha(X)\in \mathrm{S}(n,R[X],I[X])$ be such that $\alpha(0)=\mathrm{I}_n$, where in the linear case $n\ge 3$, and in the symplectic case $n\ge 6$. Then $$\alpha(X)\in \E({n},R[X],I[X]).$$

	\ec	
		
	\proof It follows from Theorem \ref{relative local global} that it is enough to prove $\alpha_{\mm}(X)\in \E({n},R_{\mm}[X], I_{\mm}[X])$ for all maximal ideals $\mm$ of $R$. However, this follows from Theorem \ref{relative vorst theorem}. This concludes the proof.\qed

\section{Injective stability for relative linear $\mathrm{K_1}$-groups}\label{affine algebra}

This section is devoted to study the relative injective stability for linear $\rm K_1$-groups of affine algebras over various base fields. We begin with a criterion for improved injective stability for relative linear and symplectic $\mathrm{K}_1$-groups over affine algebras. The proof is a formalization of the argument outlined in \cite[Theorem 3.4]{rvdk}. Before that, we recall the following definition.

\bd
A matrix $\sigma \in \mathrm{S}(n,R,I)$ is said to be homotopic to the identity relative to $I$ if there exists $\sigma(X) \in \mathrm{S}(n,R[X], I[X])$ such that $\sigma(0) = \mathrm{I}_{n}$ and $\sigma(1) = \sigma$. If $I = R$ and the above condition holds, we simply say that $\sigma$ is homotopic to the identity.

\ed

\bp
\label{cis}
Let $R$ be an affine algebra  over a field $k$ and $I\subset R$ be an ideal. Let $i\in \{1,2\}$ be such that $i=1$ in the linear case and $i=2$ in the symplectic case. Let $\alpha\in \mathrm{S}({n},R,I) \cap \E({n+i},R,I)$. If \begin{equation*}\label{H1}\tag{H}
\Um_{n+i}(R[X], \ld X^2-X \rd I[X])= e_1 \mathrm{S}({n+i},R[X], \ld X^2-X \rd I[X]),\end{equation*} then $\alpha$ is homotopic to identity relative to $I$. Further, assume that (a) $k$ is an infinite field; (b) $n\ge 3$ in the linear case, and $n\ge 6$ in the symplectic case; and (c) both $R$ and $R/I$ are regular, then $\alpha\in \E(n,R,I)$, i.e., in particular $$\mathrm{S}({n},R,I) \cap \E({n+i},R,I)=\E({n},R,I),$$where $i=1$ in the linear case and $i=2$ in the symplectic case. 
\ep

\proof Let $\alpha \in \mathrm{S}({n},R,I)\cap \E({n+i},R,I)$. Then we have $\mathrm{I}_i\perp \alpha \in \E({n+i},R,I)$, where $\mathrm{I}_i$ is the identity matrix of order $i$. Hence there exists $\gamma(X)\in \E({n+i},R[X], I[X])$ such that $\gamma(0)= \I_{n+i}$ and $\gamma(1)= \I_i\perp \alpha$. We consider the unimodular row $$v(X)= e_1\gamma(X)\in \Um_{n+i}(R[X], \ld X^2-X \rd I[X]).$$ From hypothesis, we get that $v(X)\in e_1 \mathrm{S}({n+i},R[X], \ld X^2-X \rd I[X])$. Therefore, we have $v(X)= e_1 \gamma(X)= e_1 \beta(X)$ for some $\beta(X)\in \mathrm{S}({n+i},R[X], \ld X^2-X \rd I[X])$. Then the matrix $\gamma(X) \beta(X)^{-1}$ takes the following form:
\begin{compactenum}[]
	\item 	$\begin{pmatrix}
		1 &0\\
		*& \alpha(X)
	\end{pmatrix}, \text{ for some }\alpha(X)\in \SL_n(R[X]) \text{ in the linear case or}$ 
	\item $\begin{pmatrix}
		1&0&0\\
		*&1&*\\
		*&0& \alpha(X)
	\end{pmatrix}, \text{ for some }\alpha(X)\in \Sp_{2m}(R[X])  \text{ in the symplectic case, where } n=2m.$
\end{compactenum}
Moreover, since $\beta(X)\in \mathrm{S}({n+i},R[X], \ld X^2-X \rd I[X])\subset \mathrm{S}({n+i},R[X], I[X])$, it follows that $\alpha(X) \in \mathrm{S}({n},R[X], I[X])$.  Further, one may observe that $\alpha(0)= \I_{n}$ and $\alpha(1)= \alpha$. Hence $\alpha(X)$ is the required homotopy matrix.

Now we assume (a), (b), and (c). Let $\mathfrak{p}$ be a prime ideal of $R$. Then we have $\alpha_{\mathfrak{p}}(X)\in \mathrm{S}({n},R_{\mathfrak{p}}[X], I_{\mathfrak{p}}[X])$. Then by Theorem \ref{relative vorst theorem}, we have $\alpha_{\mathfrak{p}}(X)\in \E(n,R_{\mathfrak{p}}[X], I_{\mathfrak{p}}[X])$. This holds for every prime ideal $\mathfrak{p}$ of $R$. Therefore, applying relative Local-Global principle \ref{relative local global} we obtain that $\alpha(X)\in \E({n},R[X], I[X])$. Therefore, we have $\alpha= \alpha(1) \in \E({n},R, I)$. This concludes the proof.\qed

\rmk\label{pciac} It follows from Lemma \ref{principal implies general} that if $\Um_{n+i}(R[X], \langle f \rangle)= e_1 \mathrm{S}(n+i, R[X], \langle f \rangle)$ for all $f \in \ld X^2-X\rd I[X]$, then hypothesis (\ref{H1}) holds.

The next proposition is due to Gupta and is implicit in \cite[Theorem 5.4]{AG1}.

\bp\label{cis'}
With the same notation as in Proposition \ref{cis}, in addition, assume that $I=\langle a \rangle$ is a principal ideal for some $a\in R$. Let $\alpha\in \mathrm{S}({n},R,I) \cap \E({n+i},R,I)$. If \begin{equation*}\label{H1*}\tag{H'}
	\Um_{n+i}(R[X], \ld X^2-X a\rd)= e_1 \mathrm{S}({n+i},R[X], \ld X^2-X a\rd ),\end{equation*} then $\alpha$ is homotopic to identity relative to $I$. Further, assume that (a) $R$ is regular, and (b) $n\ge 3$ in the linear case, $n\ge 6$ in the symplectic case, then $\alpha\in \E(n,R,I)$, i.e., in particular $$\mathrm{S}({n},R,I) \cap \E({n+i},R,I)=\E({n},R,I),$$where $i=1$ in the linear case and $i=2$ in the symplectic case. 
\ep

\proof


Let $\alpha \in \mathrm{S}({n},R,I)\cap \E({n+i},R,I)$. Since $\mathrm{I}_i\perp \alpha \in \E({n+i},R,I)$, it is of the form  $\prod \eta_{ij}(x)\,\eta_{kl}(\lambda a)\,\eta_{ij}(-x)$. We define $$\gamma(X):=\prod \eta_{ij}(x)\,\eta_{kl}(\lambda X)\,\eta_{ij}(-x)\in \E({n+i},R[X]).$$ Then we get that $\gamma(0)= \I_{n+i}$, $\gamma(a)= \I_i\perp \alpha$, and $$v(X):=e_1\gamma(X)\in e_1 \Um_{n+i}(R[X], \ld X^2-Xa \rd ).$$ Using (\ref{H1*}) there exists an $\beta(X)\in \mathrm{S}({n+i},R[X], \ld X^2-Xa \rd )$ such that $v(X)=e_1 \beta(X)$. We can now construct  $\alpha(X) \in \mathrm{S}({n},R[X], \ld X^2-Xa \rd)$ as done in Proposition \ref{cis}.  Further, one may observe that $\alpha(0)= \I_{n}$ and $\alpha(a)= \alpha$. Hence $$\theta(X):=\alpha(Xa)\in  \mathrm{S}({n},R[X],\ld a^2X(X-1)\rd)\subset  \mathrm{S}({n},R[X], I[X])\cap  \mathrm{S}({n},R[X], \ld X\rd)$$ is the required homotopy matrix.
	
	Now, if $R$ is regular, then by \cite[Theorem 3.3]{Vorst} and~\cite[Theorem 3.8]{br}, we obtain that $\alpha(X)\in \mathrm{E}({n},R[X], \ld X\rd)$ (see also \cite[Theorem 5.3]{AG1}), i.e., $\alpha=\alpha(a)\in \E(n,R,I)$. This concludes the proof.\qed
	
%

    Now we present the main theorem of this section.
	
	\bt\label{Relative Stability for affine}
	Let $R$ be a regular affine algebra over an infinite field $k$ of dimension $d\ge 2$. Let $I\subset R$ be an ideal such that $R/I$ is regular. Then we have the following.
	
	\begin{compactenum}[\quad\quad\quad(1)]
		\item If for any prime $p\le d$ either $p\neq \rm{char}(k)$, and $\rm{c.d._p} (k)\leq 1$; or $p=\rm{char}(k)$ and $k$ is perfect, then $\mathrm{SL}_{d+1}(R,I)\cap \E_{d+2}(R,I)= \E_{d+1}(R,I)$.
		
		\item If $k=\mathbb R$ and the set of all real points $X(\mathbb R)=\emptyset$, where $X=\Spec(R)$, then $\mathrm{SL}_{d+1}(R,I)\cap \E_{d+2}(R,I)= \E_{d+1}(R,I)$.
		\item If $k=\ol k$ and $d\ge 4$ such that $\frac{1}{d!}\in k$, then $\mathrm{SL}_{d}(R,I)\cap \E_{d+1}(R,I)= \E_{d}(R,I)$.
	\end{compactenum}
	
	\et
	
	\proof In each case it is enough to show that the hypothesis (\ref{H1}) of Proposition \ref{cis} is satisfied. We only show this. 
	
	\begin{compactenum}[(1)]
		\item If for any prime $p\le d$ either $p\neq \rm{char}(k)$, and $\rm{c.d._p} (k)\leq 1$; or $p=\rm{char}(k)$ and $k$ is perfect, then this follows from Lemma \ref{Rao-vdK}.
		\item  If $k=\mathbb R$ such that $X(\mathbb R)=\emptyset$, then this follows from Lemmas \ref{Polynomial of P is P} and \ref{completion over real affine relative}.
		\item If $k=\ol k$ and $d\ge 4$ such that $\frac{1}{d!}\in k$, then this follows from Proposition \ref{relative coordinate power over alg closed}. \qed
		
	\end{compactenum}

\rmk\label{Iisprincipal}  If $I$ is a principal ideal, then we can conclude Theorem \ref{Relative Stability for affine}, even without the hypotheses that (i) $k$ is an infinite field and (ii) $R/I$ is regular, using Proposition \ref{cis'} instead of \ref{cis} in the proof. This reproves \cite[Theorem 5.4]{AG1} and \cite[Theorem 5.3]{Banerjee2025}.

\section{Relative elementary symplectic Witt groups}\label{relative elementary symplectic Witt groups}

The purpose of this section is to establish a relative version of Karoubi periodicity sequence (cf. \cite[Section 2.1]{Syed2024}). To this end, we introduce the relative elementary symplectic Witt groups. We begin by recalling some basic facts concerning the Pfaffian of an alternating matrix; for its definition, we refer to \cite[Chapter XV, \S 9]{Lang2002}. We also recall below some well-known facts, whose proofs are classical in nature.
 
 \rmk 
 \label{properties of Pfaffian}
Let $\alpha \in \M_{2n}(R)$ be an alternating matrix. We denote the Pfaffian of $\alpha$ by $\operatorname{Pf}(\alpha)$. The Pfaffian of an alternating matrix satisfies the following properties:
 
 \begin{compactenum}[\quad\quad(1)]
 	\item $\operatorname{Pf}(\alpha\perp \beta)=\operatorname{Pf}(\alpha) \operatorname{Pf}(\beta)$,
 	\item $\operatorname{Pf}(\alpha^T)=(-1)^n \operatorname{Pf}(\alpha)$,
 	\item $\operatorname{Pf}(\alpha)^2=\operatorname{det}(\alpha)$,
 	\item $\operatorname{Pf}(\beta^T \alpha \beta)=\operatorname{det}(\beta) \operatorname{Pf}(\alpha)$, for any $\beta\in \M_{2n}(R)$,
 	\item $\operatorname{Pf}(\chi_n)=1$.
 \end{compactenum}

 
 
 
 
 
In the following we define the relative elementary symplectic Witt groups. 
 \bd\label{relative witt gr}
 Let $I$ be an ideal of $R$. Let $\A'_{2n}(R)$ be the set of all alternating matrices in $\GL_{2n}(R)$, and let $\A_{2n}(R)$ be the set of all alternating  matrices in $\GL_{2n}(R)$ with Pfaffian $1$. We define $\A'_{2n}(R,I):=\{ \alpha \in \A'_{2n}(R): \alpha- \chi_n \in \M_{2n}(I) \}$. For $m<n$, the set $\A'_{2m}(R,I)$ can be embedded into $\A'_{2n}(R,I)$ by the inclusion map $i_{m,n}:\A'_{2m}(R,I)\hookrightarrow \A'_{2n}(R,I)$ given by $ \alpha\mapsto \alpha\perp\chi_{n-m} $. We define $ \A'(R,I):=\varinjlim \A'_{2n}(R,I)$.
 One can also define the set $\A_{2n}(R,I)$ as the subset of $\A'_{2n}(R,I)$ consisting of alternating  matrices of Pfaffian $1$. We put $\A(R,I):= \varinjlim \A_{2n}(R,I)$. There is an equivalence relation $\sim_I$ on $\A'(R,I)$ and on $\A(R,I)$ as follows:
 
 Let $ \alpha \in \A'_{2m}(R,I)$ (or $\alpha\in \A_{2m}(R,I)$) and $\beta \in \A'_{2n}(R,I)$ (or $\beta\in \A_{2n}(R,I)$). Then $\alpha \sim_I \beta$ if and only if there exists $t\in \mathbb{N}$ and $\epsilon \in \E_{2(m+n+t)}(R,I)$ such that
 \begin{center}
 	$\alpha \perp \chi_{n+t}=\epsilon^T(\beta \perp \chi_{m+t})\epsilon$.
 \end{center} 
 We denote the set of all equivalence classes $\A'(R,I)/\sim_I$ by $W^{\prime}_E(R,I)$ and the set of all equivalence classes $\A(R,I)$ by $W_E(R,I)$. Whenever $I = R$, with an abuse of notation we denote $W^{\prime}_E(R,R)$ and $W_E(R,R)$ by $W^{\prime}_E(R)$ and $W_E(R)$, respectively.
 \ed

\bd Recall that, in Definition \ref{general lift} we defined the lift of a matrix in $\M_n(I)$, where $I \subset R$ is an ideal. Let $\alpha \in \A'_{2n}(R,I)$. Then $\alpha = \chi_n + \gamma$ for some $\gamma \in \M_n(I)$. We define the lift $\alpha_L \in \M_n(R \oplus I)$ of $\alpha$ by $\alpha_L := \chi_n + \gamma_L$, where $\chi_n$ is the canonical lift of $\chi_n$. One can check that $\alpha_L \in \A'_{2n}(R \oplus I, 0 \oplus I)$. Moreover, if $\alpha \in \A_{2n}(R,I)$, then $\alpha_L \in \A_{2n}(R \oplus I, 0 \oplus I)$.

\ed
 \remark
\label{inverse in Witt group}
 The matrix $\sigma_{n}$ is a symmetric matrix with $\sigma_n^{-1}=\sigma_n$. In addition, we have $\sigma_n \chi_n \sigma_n= \chi_n^{-1}$. It follows from  \cite[Chapter I, \S 3]{SV} that in   $W_E(R)$ we have   $[\alpha]^{-1}= [\sigma_n \alpha^{-1} \sigma_n]$ for any $\alpha\in \A_n(R,R)$.

 \bl
 \label{tilde operation}
 Let $\alpha\in \A'_{2n}(R,I)$ and $\beta\in \A'_{2m}(R,I)$. Then we have the following.
 \begin{compactenum}[\quad\quad\quad (a)]
     \item $({\alpha\perp \beta})_L= {\alpha}_L\perp {\beta}_L$, and
     \item $\sigma_{n} \alpha^{-1} \sigma_n\in\A'_{2n}(R,I)$ and $({\sigma_{n} \alpha^{-1} \sigma_n})_L= \sigma_{n} ({\alpha}_L)^{-1} \sigma_{n}$.
 
 \end{compactenum} 
 \el
 
 \proof Since $\alpha\equiv \chi_n$ and $\beta \equiv \chi_m\pmod{I}$, it follows that $\alpha \perp \beta \in \A'_{2(m+n)}(R,I)$. Let $\alpha= \chi_{n}+\gamma$ and $\beta=\chi_m+\tau$, where $\gamma\in \M_{2n}(I)$ and $\tau\in \M_{2m}(I)$. Then we observe the following.
 \[
 \begin{aligned}
 (\alpha\perp \beta)_L&=\chi_{(n+m)}+ (\gamma\perp \tau)_L\\
 	&= \chi_{(n+m)}+ \gamma_L\perp \tau_L\\
    &=(\chi_{n}+ \gamma_L)\perp (\chi_m+\tau_L)\\
    &=\alpha_L\perp \beta_L
 \end{aligned}
 \]
This proves (a). To prove (b) we note that since $\alpha\equiv \chi_n \pmod{I}$ and $\sigma_n \chi_n \sigma_n=\chi_n^{-1}$, one may have the following. $$\sigma_n \alpha^{-1} \sigma_n\equiv \chi_n\pmod{I}$$Implying that $\sigma_n \alpha^{-1} \sigma_n\in \A'_{2n}(R,I)$. Let the map $v : D(R, I) \to R \oplus I$, sending $(x, y)\mapsto (x, y - x)$ induce the ring homomorphism $v_*:\M_n(D(R, I))\to \M_n(R\oplus I)$. In addition, one can check that $v_*(\chi_n,\alpha)=\alpha_L$, and $v_*(\sigma_n,\sigma_n))= \sigma_n$. Since $(\chi_n, \alpha)(\chi_n^{-1}, \alpha^{-1})=(\I_n, \I_n)$ in $\GL_{2n}(D(R,I))$, we observe that $$({\alpha}_L)^{-1}= v_*(\chi_n^{-1}, \alpha^{-1}) \text{ in } \GL_{2n}(R\oplus I).$$Then we have the following.
 \[
 \begin{aligned}
 ({\sigma_{n} \alpha^{-1} \sigma_n})_L
&=v_*(\chi_n, \sigma_n \alpha^{-1} \sigma_{n})\\
&=v_*(\sigma_n \chi_n^{-1} \sigma_n, \sigma_n \alpha^{-1} \sigma_{n})\\
 	&=v_*(\sigma_n, \sigma_n) v_* (\chi_n^{-1}, \alpha^{-1}) v_*(\sigma_{n}, \sigma_{n})\\
 	&=\sigma_n (\alpha_L)^{-1} \sigma_n
 \end{aligned}
 \]
This concludes the proof.\qed 
 
In the next result, we prove that $(W'_E(R,I),\perp)$ is an abelian 
group. 
 \bp
 \label{W_E'(R,I) is abelian}
 Let $R$ be a ring and $I$ be an ideal of $R$. Then the set $W'_E(R,I)$ is an abelian group with respect to the operation $\perp$.
 \ep
 
\proof First, we observe that associativity of $\perp$ follows directly from the Definition \ref{relative witt gr}. Therefore, it suffices to show that for any $[\alpha],[\beta]\in 
W'_E(R,I)$ we have (i) $[\chi_1]$ is the identity element of $W'_E(R,I)$, (ii)  $[\alpha]$ has an inverse, and (iii) $[\alpha\perp \beta]=[\beta\perp \alpha]$. To establish this, we frequently use the fact that $W'_E(R\oplus I)$ is an abelian group. 

We first note that it follows from the Definition \ref{relative witt gr} that \begin{equation}\label{left identity}
    [\alpha \perp \chi_1]=[\alpha] \text{ in } W'_E(R,I).
\end{equation}
Since $\alpha \equiv \chi_n \pmod I$, using Remark \ref{inverse in Witt group} we have $$\sigma_n\alpha^{-1} \sigma_n\equiv \sigma_n \chi_n^{-1} \sigma_n\equiv\chi_n\pmod I,$$and hence $\sigma_n \alpha^{-1} \sigma_n\in \A'_{2n}(R,I)$. Now, we consider the matrix $$\delta=\alpha\perp \sigma_n \alpha^{-1} \sigma_n\in \A'_{4n}(R,I).$$ Let ${\delta}_L\in \A'_{4n}(R\oplus I)$ be the lift of $\delta$. It follows from Lemma \ref{tilde operation} that ${\delta}_L={\alpha}_L\perp \sigma_n{\alpha}^{-1}_L \sigma_n$. Hence again by Remark \ref{inverse in Witt group}, we get that $[\delta_L]=[\chi_1]$ in $W'_E(R\oplus I)$. Meaning, there exists $t\in \mathbb{N}$ and $\varepsilon\in \E_{4n+2t}(R\oplus I)$ such that $$\varepsilon^T ({\delta}_L\perp \chi_{t})\varepsilon =\chi_{2n+t}.$$ Let ‘bar' denote going modulo $0\oplus I$. Since ${\delta}_L\in \A'_{4n}(R\oplus I)$, going modulo $0\oplus I$, we get the following.$$\ol\varepsilon^T\chi_{2n+t}\ol{\varepsilon}=\chi_{2n+t}$$ Since $0\oplus I$ is a split ideal of $R\oplus I$, using Lemma \ref{K_2-lemma}, without loss of generality we may replace $\varepsilon$ by $\varepsilon \overline{\varepsilon}^{-1}$ and assume that $\varepsilon\in \E_{4n+2t}(R\oplus I,0\oplus I)$. We consider the projection map $\pi:R\oplus I\to R$ given by $(r,i)\mapsto r+i$, and let $\varepsilon_1=\pi(\varepsilon)\in \E_{4n+2t}(R, I)$. Under the map $\pi$, we obtain that $$\varepsilon_1^T(\alpha\perp \sigma_n\alpha^{-1} \sigma_n\perp \chi_{t})\varepsilon_1=\chi_{2n+t}.$$ This proves that \begin{equation}\label{right inverse}
    [\alpha\perp\sigma_n\alpha^{-1}\sigma_n]=[\chi_1]\text{ in }W'_E(R,I). 
\end{equation}Using a similar argument, one can also prove that \begin{equation}\label{left inverse}
    [\sigma_n\alpha^{-1}\sigma_n\perp\alpha]=[\chi_1]\text{ in }W'_E(R,I).
\end{equation}

Combining \ref{left identity}, \ref{right inverse} and \ref{left inverse} we obtain that $$[\chi_1\perp \alpha]=[(\alpha\perp\sigma_n\alpha^{-1}\sigma_n)\perp \alpha]=[\alpha\perp(\sigma_n\alpha^{-1}\sigma_n\perp \alpha)]=[\alpha\perp\chi_1]=[\alpha]\text{ in }W'_E(R,I).$$This concludes that $W'_E(R,I)$ is a group.

Now we prove that $W'_E(R,I)$ is an abelian group. Let us choose $[\alpha],[\beta]\in W'_E(R,I)$. Then $\alpha\in \A'_{2n}(R,I)$ and $\beta\in \A'_{2m}(R,I)$ for some $n,m\in \mathbb N$. Let us define\begin{equation}\label{equ:1}
    \gamma:=\alpha \perp \beta \perp \sigma_{n+m}( \beta^{-1}\perp \alpha^{-1})\sigma_{n+m}\in \A'_{4n+4m}(R, I), 
\end{equation}then $[\gamma] \in W'_E(R,I)$. Let $\alpha_L\in \A'_{2n}(R\oplus I)$ and $\beta_L\in \A'_{2m}(R\oplus I)$ be the lifts of $\alpha$ and $\beta$, respectively. One may note that, in $W'_E(R\oplus I)$, we have $$[\gamma_L]=[\alpha_L \perp \beta_L\perp \sigma_{n+m}(\beta_L^{-1}\perp \alpha_L^{-1})\sigma_{n+m}]=[\chi_1],$$ where ${\gamma}_L\in \A'_{4n+4m}(R\oplus I)$ is the lift of $\gamma$. Hence there exist $s\in \mathbb{N}$ and $\varphi\in \E
 _{j}(R\oplus I)$} such that $$\varphi^T(\gamma_L\perp \chi_{s})\varphi=\chi_{2n+2m+s},$$ where $j=4n+4m+2s$. Furthermore, if necessary, replacing $\varphi$ by $\varphi\overline{\varphi}^{-1}$, we may assume, without loss of generality, that $\varphi\in \E_{j}(R\oplus I,0\oplus I)$. Now applying the homomorphism $\pi$, we have $$\varphi_1^T(\gamma \perp \chi_{s})\varphi_1=\chi_{2n+2m+s},$$ where $\varphi_1=\pi(\varphi)\in \E_{j}(R,I)$. Hence $[\gamma]=[\chi_1]$ in $W'_E(R,I)$, i.e., $$[\alpha\perp\beta]=[\beta\perp \alpha] \text{ in } W'_E(R,I).$$This concludes the proof.\qed

 \bc
 \label{W_E(R,I) is abelian}
 Let $R$ be a ring and $I$ be an ideal of $R$. Then the set $W_E(R,I)$ is an abelian group with respect to the operation $[\alpha] [\beta]=[\alpha \perp \beta]$ 
 \ec
 
 \proof Since $W_E(R,I) \subset W_E'(R,I)$, it is enough to show that if $[\alpha]$ and $[\beta]\in W_E(R,I)$, then $[\alpha\perp \beta]\in W_E(R,I)$ and $ [\alpha]^{-1}\in W_E(R,I)$.

 Let $[\alpha], [\beta]\in W_E(R,I)$, then we have $\operatorname{Pf}(\alpha)=\operatorname{Pf}(\beta)=1$. By Remark \ref{properties of Pfaffian}, we have $\operatorname{Pf}(\alpha \perp \beta)=\operatorname{Pf}(\alpha) \operatorname{Pf}(\beta)=1$. Hence $[\alpha \perp \beta]\in W_E(R,I)$. 
 
 Now for any $\alpha\in \A'_{2n}(R,I)$, by Proposition \ref{W_E'(R,I) is abelian}, we have $[\alpha]^{-1}= [\sigma_n \alpha^{-1} \sigma_{n}]$ in $W'_E(R,I)$. Since $\operatorname{Pf}(\alpha)=1$, applying Remark \ref{properties of Pfaffian}, we get that $\operatorname{det}(\alpha^{-1})=1$. Then we have the following:
\begin{align*}
	\operatorname{Pf}(\alpha^{-1})&=(-1)^n \operatorname{Pf}((\alpha^{-1})^T)\\
	&=(-1)^n \operatorname{det}(\alpha^{-1}) \operatorname{Pf}(\alpha)\\
	&=(-1)^n 
\end{align*}
 Once again, applying Remark \ref{properties of Pfaffian}, we obtain that $\operatorname{Pf}(\sigma_n \alpha^{-1} \sigma_n)= \operatorname{det}(\sigma_n) \operatorname{Pf}(\alpha^{-1})=1$. Implying that $[\alpha]^{-1}\in W_E(R,I)$. Hence, being a subgroup of $W_E'(R,I)$, the group $W_E(R,I)$ is an abelian group. This concludes the proof. \qed 

We call $W_E(R,I)$ the relative elementary symplectic Witt group of 
$R$ with respect to the ideal $I$.
 For any ideal $I\subset R$, let $C$ be the kernel of the group homomorphism  $R^{\times}\to (R/I)^{\times}$ induced by the canonical projection $R\to R/I$. For any $\alpha\in \A'(R,I)$ we have $\operatorname{Pf}(\alpha)-1\in I$. Therefore, the Pfaffian map $\operatorname{Pf}: \A'(R,I)\to R^{\times}$ induces a map $\operatorname{Pf}:W_E'(R,I)\to C$.
 
 \bl
 \label{spit exact sequence}
The following sequence is split exact

 \[
 1\rightarrow W_E(R,I)\xrightarrow{j} W_E'(R,I)\xrightarrow{\operatorname{Pf}} C\rightarrow 1,
 \]
 where $j$ is the natural inclusion map. \el
 
 \proof First, we note that the map $j$ is injective by Corollary \ref{W_E(R,I) is abelian}. Further, the sequence is exact in the middle, as the kernel of $\operatorname{Pf}$ is precisely $W_E(R,I)$. Hence to prove the sequence is exact, we only need to show that $\operatorname{Pf}$ is surjective. However, for any $a\in C$, the matrix $\alpha_a:=\begin{pmatrix}
 	0&a\\-a&0
 \end{pmatrix}\in \A'_{2}(R,I)$ and $\operatorname{Pf}(\alpha_a)=a$. This proves that the map $\operatorname{Pf}$ is surjective. 
 
 Now we define $u: C\to W_E'(R,I)$ by $u(a)= [\alpha_a]$. If $u$ is a group homomorphism, then $u$ is a split homomorphism, as $ \operatorname{Pf}\circ u(a)= a$ for all $a\in C$. In the following we prove that $u$ is a homomorphism.
 
 Let $a,b\in C$. Then $u(a)=[\alpha_a]$ and $u(b)=[\alpha_b]$ and $u(ab)= [\alpha_{ab}]$. Therefore, we are done if we show that $[\alpha_{ab}]=[\alpha_a\perp \alpha_b]$. Since $b\in C$, we have $\gamma_b:=\begin{pmatrix}
 b&0\\0&1
 \end{pmatrix}\in \GL_2(R,I)$. Applying Whitehead lemma \cite[Corollary 2.3]{SV}, it follows that $\gamma_b^{-1}\perp \gamma_b\in \E_4(R,I)$. Hence we obtain the following:
 \begin{align*}
 	(\gamma_b^{-1}\perp \gamma_b)^T(\alpha_{ab}\perp \chi_1) (\gamma_b^{-1}\perp \gamma_b)&=\alpha_a \perp \alpha_b
 \end{align*}
 Therefore $u(ab)=[\alpha_{ab}]=[\alpha_{ab}\perp \chi_1]=[\alpha_a \perp \alpha_b]=u(a)u(b)$. This completes the proof.\qed
 
    We now observe that, for any ideal $I \subset R$, there is a natural map $i:W_E(R,I) \to W_E(R)$. In the next theorem we show that there exists the following homology sequence $$W_E(R,I)\xrightarrow{i} W_E(R)\xrightarrow{q_*} W_E(R/I)$$for Witt groups, where $q_*$ is induced from the canonical projection $q:R\to R/I$. Furthermore, if $I$ is a split ideal of $R$, i.e., when $q: R \to R/I$ is a retract, then the homology sequence reduces to the following short exact sequence.$$1 \rightarrow W_E(R,I)\xrightarrow{i} W_E(R)\xrightarrow{q_*} W_E(R/I)\rightarrow 1$$
    
 \bt
 \label{split implies injective}
 Let $R$ be a ring  and $I$ be a split ideal of $R$. Then we have the following short exact sequence.
 \[
 1 \rightarrow W_E(R,I)\xrightarrow{i} W_E(R)\xrightarrow{q_*} W_E(R/I)\rightarrow 1
 \]

 \et
 
 \proof First, we show that $i$ is injective. Let $[\alpha]\in W_E(R,I)$ be such that $i([\alpha])=[\chi_1]$. We may assume that $\alpha\in \A_{2n}(R,I)$. Since  $[\alpha]=[\chi_1]$ in $W_E(R)$, there exist $t\in \mathbb{N}$ and $\varepsilon\in \E_{2(n+t)}(R)$ such that
 \[
 \varepsilon^T(\alpha\perp \chi_{t})\varepsilon=\chi_{n+t}.
 \]
 Let `bar' denote going modulo $I$. Then we have $$\bar{\varepsilon}^T \chi_{n+t} \bar{\varepsilon}=\chi_{n+t}.$$ Since $I$ is a split ideal there exists a map $s:R/I\to R$ such that $q\circ s=id$. Then one may observe that $s(\bar{\varepsilon})\in \E_{2(n+t)}(R)$. Further, since $I$ is a split ideal, by Lemma \ref{K_2-lemma} we get that $$\varepsilon s(\bar{\varepsilon})^{-1}\in \SL_{2n+2t}(R,I)\cap \E_{2n+2t}(R)= \E_{2n+2t}(R,I).$$ Then we have the following.
 \[
 (\varepsilon s(\bar{\varepsilon})^{-1})^T (\alpha \perp \chi_{t})(\varepsilon s(\bar{\varepsilon})^{-1})=\chi_{n+t}
 \]
 This shows that $[\alpha]=[\chi_1]$ in $W_E(R,I)$. Hence the map $i$ is injective.  

 Next we show that $\ker(q_*)={\rm im}(i)$. Let $[\alpha]\in W_E(R,I)$. We may assume that $\alpha\in \A_{2n}(R,I)$. Then we have $q_*(i([\alpha]))= q_*([\alpha])=[\chi_n]=[\chi_1]$. This shows that ${\rm im}(i)\subset \ker({q_*})$. Now let us choose $[\beta]\in W_E(R)$ such that $q_*([\beta])=[\chi_1]$. Then $\beta \in \A_{2m}(R)$ for some $m$. There exist $u\in \mathbb{N}$ and $\theta\in \E_{2m+2u}(R/I)$ such that $$\theta^T(\bar{\beta}\perp \chi_{u} )\theta= \chi_{m+u}.$$ We observe that $s(\theta)\in \E_{2m+2u}(R)$ and hence $s(\theta)^T (\beta\perp \chi_{u})s(\theta)\in \A_{2m+2u}(R)$. Therefore, we get that $$s(\theta)^T (\beta\perp \chi_{u})s(\theta)\equiv \chi_{m+u}\mod (I),$$ i.e., $s(\theta)^T (\beta\perp \chi_{u})s(\theta)\in \A_{2m+2u}(R,I)$. Thus in the group $W_E(R)$ we obtain the following.
 \[
 i([s(\theta)^T (\beta\perp \chi_{u})s(\theta)])=[s(\theta)^T (\beta\perp \chi_{u})s(\theta)]=[\beta\perp \chi_{u}]=[\beta]
 \] 	
 Hence we prove that $\ker({q_*})= {\rm im}(i)$. 
 
 In the remaining part, we show that the map $q_*$ is surjective. Let us choose $[\gamma]\in W_E(R/I)$. Then $[s(\gamma)]\in W_E(R)$ and $q_*([s(\gamma)])= [q(s(\gamma))]=[\gamma]$. Hence $q_*$ is surjective. This concludes the proof.\qed
 

We now proceed to establish a relative version of Karoubi's periodicity 
sequence. Note that, as in the absolute case, one can define a homomorphism 
$H\colon \mathrm{K}_1(R,I)\to W'_E(R,I)$ sending $[\alpha]\mapsto  
[\alpha^T \chi_n \alpha]$ for some $\alpha \in \GL_{2n}(R,I)$. In the 
following proposition we identify the kernel of $H$ with respect to an arbitrary alternating matrix. Before that, we recall the following definition.

\bd\label{symplectic matrix w.r.t. a form}
Let $ I\subset R$ be an ideal and let $\varphi\in Alt'(R,I)$. The group $\Sp_{\varphi}(R,I):=\{\alpha \in \GL_{2n}(R,I) : \alpha^T\varphi\alpha = \varphi\}$ is called symplectic group with respect to $\varphi$. Any element $\alpha\in \Sp_{\varphi}(R,I)$ is called a relative symplectic matrix with respect to $\varphi$.
\ed

The following result is the relative version of \cite[Lemma 2.1]{Syed2024}. 

 \bp
 \label{kernel of H}
Let $I\subset R$ be an ideal and $\gamma\in \A_{2n}(R,I)$. If $\alpha\in \GL_{2n}(R,I)$ such that $[\alpha]\in \ker({H})$, then there exist $m\in \mathbb{N}$ and $\alpha'\in \SL_{2n+2m}(R,I)$ such that (1) $[\alpha]=[\alpha']$ in $\rm K_1(R,I)$, and (2) $\alpha'\in \Sp_{\varphi}(R,I)$, where $\varphi=\chi_m\perp\gamma $ . 
 \ep
 
 \proof Since $H([\alpha])=[\chi_1]$, we have $[\alpha^T \chi_n \alpha]=[\chi_1]$ in $W_E(R,I)$. Hence there exists $s\in \mathbb{N}$ and $\varepsilon \in \E_{2n+2s}(R,I)$ such that the following hold. \begin{equation}\label{eqn}
 \varepsilon^T(\alpha^T \chi_n \alpha\perp \chi_{s})\varepsilon = \chi_{s+n}\end{equation} 
We set $m=n+s$, and $\alpha'=(\alpha \perp \I_{2m})(\varepsilon\perp \I_{2n})$. Then clearly $\alpha'\in \SL_{2m+2n}(R,I)$ and $[\alpha]=[\alpha']$ in $\mathrm{K_1}(R,I)$. Then it only remains to show that $\alpha'\in \Sp_{\varphi}(R,I)$. To verify this, note that \eqref{eqn} implies the following.
\begin{align*}
    \epsilon^T(\alpha^T\perp \I_{2s})\chi_m(\alpha\perp I_{2s})\epsilon\perp\gamma=\varphi
\end{align*}
However, this is nothing but
\begin{align*}
    (\alpha')^T\varphi\alpha'=\varphi.
\end{align*}Implying that $\alpha'\in \Sp_{\varphi}(R,I)$. This concludes the proof.\qed

Note that the forgetful homomorphism $\Sp_{2n}(R,I)\to \GL_{2n}(R,I)$ induces a natural homomorphism $ f:\KSp(R,I)\to \mathrm{K_1}(R,I)$.
Throughout the remainder of this article, we use the notation $f$ to denote this map, irrespective of the ideal $I$. The following result is known as relative Karoubi periodicity sequence. The absolute case of the following can be found \cite[Proposition 3.2]{FRS}.
 \bl
 \label{Relative Karoubi periodicity sequece}
Let $R$ be a ring and $I$ be an ideal of $R$. Then the following sequence is exact.
 \[
 \mathrm{K_1Sp}(R,I) \xrightarrow{f} \mathrm{K_1}(R,I) 
 \xrightarrow{H} W'_E(R,I)
 \]

 \el
 
 \proof  Let $[\alpha]\in \mathrm{K_1}(R,I)$ be such that $H([\alpha])=[\chi_1]$. Then using Proposition \ref{kernel of H} (taking $\gamma=\chi_n$ for some suitable $n\in \mathbb N$) we can find an $\alpha'\in \Sp_{2m}(R,I)$ such that $[\alpha]=[\alpha']$ in $\mathrm{K_1}(R,I)$. This shows that $[\alpha]\in {\rm im}(f)$.
 
 
 Conversely, let $[\beta]\in \KSp(R,I)$. Then $\beta \in \Sp_{2n}(R,I)$ for some $n\in \mathbb{N}$, that is, $\beta^T \chi_n \beta =\chi_n$. Thus we have $H(f([\beta]))=[\beta^T \chi_n \beta] =[\chi_n]=[\chi_1]$, and hence ${\rm im}(f)\subset \ker(H)$. This completes the proof.\qed

 \bc
 \label{Karoubi periodicity for relative}
 Let $R$ be a ring and $I$ be an ideal of $R$. Then the following sequence is exact.
 \[
 \KSp(R,I) \xrightarrow{f} \SK(R,I) 
 \xrightarrow{H} W_E(R,I)
 \]
 \ec
 
 \proof Since $\operatorname{det}(\alpha)=1$ for $\alpha\in \Sp_{2n}(R,I)$, the image of $f$ lies in $\mathrm{SK_1}(R,I)$. Also by Remark \ref{properties of Pfaffian} it follows that  $\operatorname{Pf}(\alpha^T \chi_n \alpha)= \operatorname{det}(\alpha) \operatorname{Pf}(\chi_n)=1$, for any $\alpha\in \SL_{2n}(R,I)$. Therefore, the image of $H$ lies in $W_E(R,I)$. This concludes the proof.\qed
 
%


\section{Symplectic completion of relative unimodular rows}\label{symplectic completion of unimodular rows}

In this section, we study the completion problem of a relative unimodular row 
to the first row of a relative symplectic matrix with respect to an 
arbitrary alternating form. We begin with the following preparatory 
lemma, whose absolute case can be found in \cite[Lemma 3.5]{Syed2024}

 \bl 
 \label{key ingredient}
 Let $R$ be a ring and $I$ be an ideal of $R$. Let $\theta_1, \theta_2\in \A'_{2n}(R,I)$ such that $\alpha^T(\chi_1 \perp \theta_1) \alpha= \chi_1\perp \theta_2$ for some $\alpha\in \SL_{2(n+1)}(R,I)$. Further, assume that $$\Um_{2n+2}(R,I)=e_{1}(\Sp_{\varphi}(R,I)\cap \E_{2n+2}(R,I)),$$ where  $\varphi=\chi_1\perp\theta_1$. Then there exists $\beta\in \SL_{2n}(R,I)$ such that $\beta^T \theta_1 \beta=\theta_2$ and $[\alpha]=[\beta]$ in $\mathrm{K_1}(R,I)$. 
 \el
 
 \proof Let us consider the row $e_1\alpha^T$. Then we have $$e_{1}\alpha^T \in \Um_{2n+2}(R,I)= e_{1}(\Sp_{\varphi}(R,I)\cap \E_{2n+2}(R,I)). $$ Meaning, there exists $\gamma^T\in \Sp_{\varphi}(R,I)\cap \E_{2n+2}(R,I)$ such that $e_{1}\alpha^T =e_{1}\gamma^T$. Let us take $\delta=\gamma^{-1} \alpha$. Then we have \begin{equation}\label{A}
 \delta^T(\chi_1\perp\theta_1)\delta=\chi_1\perp \theta_2 \end{equation}
 and $e_{1}\delta^T= e_{1}$. Therefore, one can write \begin{equation}\label{B}
\delta=\begin{pmatrix}
 	  1 &p &x\\
 	0&q&v\\
 	0&u^T&\beta
 \end{pmatrix}, \end{equation} where $\beta\in {\M}_{2n}(R)$, $u,v,x\in {\M}_{1,2n}(I)$, $ p\in I$, and $q\in 1+I$. 
 
 Now we show that $\beta\in \SL_{2n}(R,I)$. It follows from (\ref{A}) that $q=1, v=0$ and $\beta^T \theta_1 \beta=\theta_2$. Therefore, we get that \begin{equation}\label{C}
     \delta\equiv \I_2\perp\beta \mod\E_{2n+2}(R,I).
 \end{equation}Since $\delta\in \SL_{2(n+1)}(R,I)$, it follows that $\beta\in \SL_{2n}(R,I)$. Further, from (\ref{C}) we obtain that $[\alpha]=[\beta]$ in $\mathrm{K_1}(R,I)$. This completes the proof.\qed
    
	
	
	

Recall the definition of $\ESp_{\varphi}(R,I)$ from 
\cite[Definition~3.4]{cr1}, where $\varphi\in \A'_{2n}(R,I)$. It 
follows from the definition that $\ESp_{\varphi}(R,I)\subset 
\E_{2n}(R,I)$, and thus $\ESp_{\varphi}(R,I)$ acts on 
$\Um_{2n}(R,I)$. The orbit space of this action is denoted by 
$\Um_{2n}(R,I)/\ESp_{\varphi}(R,I)$. We recall the following result 
from \cite[Theorem~7.2]{Chattopadhyay2016}.
 
 \bl
 \label{equlaity of orbit spaces}
 Let $R$ be a ring with $R = 2R$ and $I$ be an ideal of $R$. Let $n \ge 2$ and $\varphi\in \A_{2n}(R,I)$. There is a bijective correspondence between the orbit spaces $\Um_{2n}(R, I)/\E_{2n}(R, I)$, $\Um_{2n}(R, I)/\ESp_{2n}(R, I)$ and  $\Um_{2n}(R,I)/\ESp_{\varphi}(R, I)$.  
 \el

 \bc
 \label{transitive action for ESp-alternating}
 Let $R$ be a ring of dimension $d$ with $R=2R$ and $I\subset R$ be an ideal of $R$. Let $\varphi\in \A_{2n}(R,I)$. Then for $2n\ge \max\{3, d+2\}$, we have $\Um_{2n}(R,I)= e_1 \ESp_{\varphi}(R,I)$.
 \ec
 \proof By \cite[Theorem 7.2]{SV}, we have $\Um_{2n}(R,I)= e_1 \E_{2n}(R,I)$ for $2n\ge \max\{3,d+2\}$. Therefore, by Lemma \ref{equlaity of orbit spaces}, we have $\Um_{2n}(R,I)= e_1 \ESp_{\varphi}(R,I)$.

 \bc
  \label{transitive action wrt alternating form}
 Let $R$ be a ring of dimension $d$ with $2R=R$ and $I$ be an ideal of $R$. Let $\varphi\in \A_{2n}(R,I)$. Then for $2n\ge \max\{3, d+2\}$, we have $\Um_{2n}(R,I)= e_1 (\E_{2n}(R,I)\cap \Sp_{\varphi}(R,I))$. 
 \ec
 
 \proof Since $\ESp_{\varphi}(R,I)\subset \E_{2n}(R,I)\cap \Sp_{\varphi}(R,I) \subset \E_{2n}(R,I)$, we have $$e_1\ESp_{\varphi}(R,I)\subset e_1(\E_{2n}(R,I)\cap \Sp_{\varphi}(R,I)) \subset e_1\E_{2n}(R,I).$$ However, it follows from Lemma \ref{equlaity of orbit spaces} and Corollary \ref{transitive action for ESp-alternating} that for all $2n\ge \max\{3,d+2\}$ one can have the following. $$e_1\ESp_{\varphi}(R,I)=e_1\ESp_{2n}(R,I)=\Um_{2n}(R,I)$$ This will imply that $\Um_{2n}(R,I)= e_1 (\E_{2n}(R,I)\cap \Sp_{\varphi}(R,I))$. This concludes the proof.\qed
 
 Proof of the following result is implicit in the proof  of \cite[Theorem 7.2]{Chattopadhyay2016}.
 
  \bl
 \label{relative version of Lemma 5.5 of SV}
 Let $R$ be a ring and $I$ be an ideal of $R$ and $\varphi\in \A_{2n}(R,I)$. Then for $n\geq 2$, we have $e_1 \E_{2n}(R,I)= e_1 \ESp_{\varphi}(R,I)$. 
 \el

 In the following proposition, we give a sufficient condition for the symplectic completion of a relative unimodular row with respect to an arbitrary alternating matrix. This result will be used crucially in the remainder of the section. 
 
 

 \bp
 \label{sufficient condition for alternating completion}
 Let $R$ be a ring and $I$ be an ideal of $R$. Let $\gamma\in \A_{2n}(R,I)$, and $v\in \Um_{2n}(R,I)$, where $n\ge 2$. Suppose that the following conditions hold.
 \begin{enumerate}
     \item The canonical map $\frac{\SL_{2n}(R,I)}{\E_{2n}(R,I)}\to \SK(R,I)$ is injective.
     \item There exists $\alpha\in \SL_{2n}(R,I)$ such that (a) $v=e_1 \alpha$ and (b) $\alpha\in {\rm im}(f)$, where $f:\KSp(R,I)\to \SK(R,I)$.
       \item $\Um_{2n+2i}(R,I) = e_1\big(\E_{2n+2i}(R,I) \cap \Sp_{\tau}(R,I)\big)$ for all $i\geq 1$, where $\tau =\chi_i\perp \gamma $.
 \end{enumerate} Then $v\in e_1 \Sp_{\gamma}(R,I)$.
 \ep
 
 \proof It follows from Corollary \ref{Karoubi periodicity for relative} that ${\rm im}(f)=\ker(H)$, and hence $\alpha \in \ker(H)$. By Proposition \ref{kernel of H}, there exist $m \in \mathbb{N}$ and $\beta \in \SL_{2m+2n}(R,I)$ such that $[\alpha]=[\beta]$ in $\mathrm{K_1}(R,I)$ and $\beta \in \Sp_{\varphi}(R,I)$, where $\varphi=\chi_m\perp \gamma$. From hypothesis (3) we obtain that
\[
\Um_{2n+2m}(R,I) = e_1\big(\E_{2n+2m}(R,I) \cap \Sp_{\varphi}(R,I)\big).
\]
In view of Lemma \ref{key ingredient}, we may further assume without loss of generality that $\beta\in \Sp_{\gamma}(R,I)$. Since, by hypothesis (1), the canonical map
\[
\frac{\SL_{2n}(R,I)}{\E_{2n}(R,I)} \longrightarrow \SK(R,I)
\]
is injective, it follows that $\alpha \beta^{-1} \in \E_{2n}(R,I)$. Now, applying Lemma \ref{relative version of Lemma 5.5 of SV}, we find $\beta' \in \ESp_{\gamma}(R,I)$ such that
\[
e_1 \alpha \beta^{-1} = e_1 \beta'.
\]
Therefore, we obtain the following.
\[
v = e_1 \alpha = e_1 \beta' \beta \in e_1 \Sp_{\gamma}(R,I)
\]
This completes the proof.\qed




Note that there is a canonical map $i:\mathrm{K_1}(R,I)\to \mathrm{K_1}(R)$ obtained by forgetting the relative structure. In the following result, we show that to determine whether an element lies in the image of the map $f:\KSp(R,I)\to \mathrm{K_1}(R,I)$, it suffices to verify this in the absolute case, whenever $I$ is a split ideal.
 
 \bl
 \label{key result-2}
 Let $R$ be a ring and $I$ be a split ideal of $R$. If $[\alpha]\in \mathrm{K_1}(R,I)$ is such that $i([\alpha])\in \mathrm{K_1}(R)$ lies in the image  of $f: \KSp(R)\to \mathrm{K_1}(R)$, then $[\alpha]$ lies in the image of $f: \KSp(R,I)\to \mathrm{K_1}(R,I)$.  
 \el	
 
 \proof We have the following commutative diagram. 
 \[
 \begin{tikzcd}
 	\mathrm{K_1}(R,I) \arrow[r, "i"] \arrow[d, "H"'] & \mathrm{K_1}(R) \arrow[d, "H"] \\
 	W_E(R,I) \arrow[r, "i"'] & W_E(R)
 \end{tikzcd}
 \] 
 Since $i([\alpha])\in {\rm im}(f)$, it follows from Karoubi periodicity sequence \ref{Karoubi periodicity for relative} that $H(i([\alpha]))=[\chi_1]$. By the commutativity of the above diagram we obtain that $i(H([\alpha]))=[\chi_1]$. Applying Theorem \ref{split implies injective}, we obtain that the map $i:W_E(R,I)\to W_E(R)$ is injective, in particular we get that $H([\alpha])=[\chi_1]$. Therefore, by Corollary \ref{Karoubi periodicity for relative} it follows that $[\alpha]$ lies in the image of $f$. This concludes the proof. \qed 
 

The following lemma is implicit in the proofs of \cite[Theorems 3.6 and 3.8]{Syed2024}.

 \bl
 \label{completion of coordinate power lies in im(f)}
  Let $R$ be an affine algebra of dimension $d$ over a perfect field $k$ such that $\operatorname{char}(k)\not= 2$. Then we have the following.
 \begin{enumerate}
     \item Assume that $d$ is an odd integer. Let $v= \big(v_1, v_2,\dots, v_d, v_{d+1}^{2 (d!)^2}\big)\in \Um_{d+1}(R)$. Then there exists $\alpha\in \SL_{d+1}(R)$ such that $v= e_1 \alpha$, and $[\alpha]\in \SK(R)$ lies in the image of $f: \KSp(R)\to \SK(R)$.
     \item Assume that $d$ is an even integer. Let $v= \big(v_1, v_2,\dots, v_{d-1}, v_{d}^{d!^2}\big)\in \Um_{d}(R)$. Then there exists $\alpha\in \SL_{d}(R)$ such that $v= e_1 \alpha$, and $[\alpha]\in \SK(R)$ lies in the image of $f: \KSp(R)\to \SK(R)$.
 \end{enumerate}
 \el

The next result is a relative version of the preceding one and is 
crucially used in the proof of the main theorem of this section.

\bp
 \label{Tariq relative}
 Let $R$ be an affine algebra of dimension $d$ over a perfect field $k$ such that $\operatorname{char}(k)\not= 2$.  Let $I$ be an ideal of $R$. Then we have the following.
 \begin{enumerate}
     \item Assume that $d$ is an odd integer. Let $v= \big(v_1, v_2,\dots, v_d, v_{d+1}^{2 (d!)^2}\big)\in \Um_{d+1}(R,I)$. Then there exists $\alpha\in \SL_{d+1}(R,I)$ such that $v= e_1 \alpha$, and $[\alpha]\in \SK(R,I)$ lies in the image of $f: \KSp(R,I)\to \SK(R,I)$.
     \item Assume that $d$ is an even integer. Let $v= \big(v_1, v_2,\dots, v_{d-1}, v_{d}^{d!^2}\big)\in \Um_{d}(R,I)$. Then there exists $\alpha\in \SL_{d}(R,I)$ such that $v= e_1 \alpha$, and $[\alpha]\in \SK(R,I)$ lies in the image of $f: \KSp(R,I)\to \SK(R,I)$.
 \end{enumerate}
 \ep

\proof We give the proof for odd $d$ only, as the case of even $d$ is verbatim. We consider the lift ${v}_L\in \Um_{d+1}(R\oplus I, 0\oplus I)\subset \Um_{d+1}(R\oplus I)$ of $v$. From Lemma \ref{completion of coordinate power lies in im(f)} it follows that there exists $\Gamma\in \SL_{d+1}(R\oplus I)$ such that $v_L=e_1\Gamma$, and $[{\Gamma}]\in \SK(R\oplus I)$ lies in the image of $f: \KSp(R\oplus I)\to \SK(R\oplus I)$. Let 'bar' denote going modulo $0\oplus I$. Via the canonical inclusion $\frac{R\oplus I}{0\oplus I}\simeq R\oplus 0\inj R\oplus I$, we may treat $\ol\Gamma$ as an element of $\SL_{d+1}(R\oplus I)$. Replacing $\Gamma$ with $\ol{\Gamma}^{-1}\Gamma$ we obtain that $\Gamma\in \SL_{d+1}(R\oplus I, 0\oplus I)$. We also note that, the replacement of $\Gamma$ retain the facts that $v_L=e_1\Gamma$ and $[{\Gamma}]\in \SK(R\oplus I)$ lies in the image of $f$.

Since $0\oplus I$ is a split ideal of $R\oplus I$, applying Lemma \ref{key result-2}, we have that $[{\Gamma}]\in \SK(R\oplus I, 0\oplus I)$ lies in the image of $f:\KSp(R\oplus I, 0\oplus I)\to \SK(R\oplus I, 0\oplus I)$. Now consider the ring homomorphism $\pi: R\oplus I \to R$, sending $(r,i)\mapsto r+i$. One may note that $\pi(v_L)=v$. Moreover, this map induces the following commutative diagram:
 
 \[
 \begin{tikzcd}
 	\KSp (R\oplus I, 0\oplus I)\arrow[r, "f"] \arrow[d, "\pi_*"'] & \SK(R\oplus I, 0\oplus I) \arrow[d, "\pi_*"] \\
 	\KSp(R,I) \arrow[r, "f"'] & \SK(R,I)
 \end{tikzcd}
 \]
Let us define $\alpha:=\pi(\Gamma)$. Then one may note that $\alpha\in \SL_{d+1}(R,I)$ and $e_1\alpha=v$. Now, the fact that $[\alpha]\in \SK(R,I)$ lies in the image of $f: \KSp(R,I)\to \SK(R,I)$ follows from the commutativity of the above diagram. This concludes the proof.\qed

 

We are now ready to prove the main result of this section. 
\bt\label{symplectic completion}
Let $R$ be a regular affine algebra of dimension $d$ over an infinite perfect field $k$ such that $\operatorname{char}(k)\not= 2$. Let $I$ be an ideal of $R$ such that $R/I$ is regular.

\begin{compactenum}
    \item Let $d\ge 3$ be an odd integer and let $\varphi\in \A_{d+1}(R,I)$. Assume that one of the following holds. 
    \begin{compactenum}
        \item Suppose that, for any prime $p\le d$ with  $\mathrm{char}~(k)\neq p$, we have $\rm{c.d._p} (k)\leq 1$.
        \item Suppose that, $k=\mathbb{R}$ and $X(\mathbb{R})=\emptyset$, where $X=\Spec(R)$.
    \end{compactenum}Then $\Um_{d+1}(R,I)= e_1 \Sp_{\varphi}(R,I)$.
    \item Let $d\ge 4$ be an even integer and let $\varphi\in \A_{d}(R,I)$. If $k$ is an algebraically closed field with $(d-1)!\in k^{\times}$, then $\Um_d(R,I)=e_1\Sp_{\varphi}(R,I)$.
    
\end{compactenum}

\et

 \proof We first give the proof in the case when $d$ is an odd integer; and $k$ and $R$ satisfy the conditions (a). Note that, since 
$e_1 \Sp_{\varphi}(R,I)\subset \Um_{d+1}(R,I)$, it suffices to show 
that $\Um_{d+1}(R,I)\subset e_1 \Sp_{\varphi}(R,I)$.

Let $v=(v_1,v_2,\dots, v_{d+1})\in \Um_{d+1}(R,I)$. We show that $v$ satisfies all the hypotheses 
of Proposition \ref{sufficient condition for alternating completion}. It follows from Theorem \ref{Relative Stability for affine} that the canonical map$$\frac{\SL_{d+1}(R,I)}{ \E_{d+1}(R,I)}\to \SK(R,I)$$ is injective. Moreover, for all $i\geq 1$, since $d+1+2i\ge \dim(R)+2$, by Corollary \ref{transitive action wrt alternating form} we have $\Um_{2n+2i}(R,I) = e_1\big(\E_{2n+2i}(R,I) \cap \Sp_{\tau}(R,I)\big)$, where $\tau = \chi_i\perp \varphi $. Then it only remains to show that  there exists $\alpha\in \SL_{d+1}(R,I)$ such that $v= e_1 \alpha$ and $[\alpha]\in \SK(R,I)$ lies in $\rm{im}(f)$, where $f: \KSp(R,I)\to \SK(R,I)$. In the following we show this.

Applying Proposition \ref{relative coordinate 
power}, one can find an $\varepsilon\in \E_{d+1}(R,I)$ such that $$v\varepsilon 
= (v_1,v_2,\dots, v_{d+1}^{2(d!)^2}).$$ Let us take $v'=(v_1,v_2,\dots, v_{d+1}^{2(d!)^2})$. Now, using Proposition \ref{Tariq relative} we get an $\alpha'\in \SL_{d+1}(R,I)$ such that $v'= e_1 \alpha'$, and $[\alpha']\in \SK(R,I)$ lies in the image of $f: \KSp(R,I)\to \SK(R,I)$. We define $\alpha:=\alpha'\epsilon^{-1}\in \SL_{d+1}(R,I)$. Then we have the following.$$e_1\alpha =e_1\alpha'\epsilon^{-1}=v'\epsilon^{-1}=v$$ Moreover, since 
$[\alpha]=[\alpha' \varepsilon^{-1}]=[\alpha']$ in $\mathrm{SK}_1(R,I)$, we get that $[\alpha]\in \SK(R,I)$ lies in $\rm{im}(f)$. This completes the proof for 1(a).

In the remaining cases, the proofs are identical. One only needs to use Propositions~\ref{relative coordinate power over real} and \ref{relative coordinate power over alg closed} in place of Proposition~\ref{relative coordinate power} in the above proof to obtain (1)(b) and (2), respectively. Hence, we omit them to avoid repetition.\qed

\rmk\label{symplectic completionRmk} It is worth noting that if $\varphi$ is the standard alternating matrix $\chi_{\star}$, then, by Lemma~\ref{principal implies general}, we may assume, without loss of generality, that $I$ is a principal ideal in the proof of Theorem~\ref{symplectic completion}. In this case, the same proof goes through even without the additional assumptions that (i) $k$ is an infinite field and (ii) $R/I$ is regular (cf. Remark \ref{Iisprincipal}). This observation is used crucially in the next section.

\section{Injective stability for relative $\mathrm{K_1Sp}$-groups}\label{laurent}

Let $R$ be a regular affine algebra of dimension $d$ over an infinite 
field $k$, and let $I$ be an ideal of $R$ such that $R/I$ is regular. 
In this section, we prove the main theorems of this article. We begin 
with the following theorem on injective stability for relative $\mathrm{K_1Sp}$-groups of regular affine algebras over various base fields.
    
    \bt\label{Relative Stability for affine 2}
	Let $R$ be a regular affine algebra over an infinite field $k$ of dimension $d\ge 2$ such that $\rm{char}(k)\not=2$. Let $I\subset R$ be an ideal such that $R/I$ is regular.
    \begin{compactenum}
\item Suppose that $d$ is an even integer and $k$ is a perfect field. 
  \begin{compactenum}
\item If $d\ge 6$ assume that for any prime $p\le d+1$ with $\rm{char}(k)\neq p$  we have $\rm{c.d._p} (k)\leq 1$, then $\mathrm{Sp}_{d}(R,I)\cap \ESp_{d+2}(R,I)= \ESp_{d}(R,I)$.

\item If $d\ge 6$, assume that $k=\mathbb R$ and $X(\mathbb R)=\emptyset$, then $\mathrm{Sp}_{d}(R,I)\cap \ESp_{d+2}(R,I)= \ESp_{d}(R,I)$.
	  \end{compactenum}

\item     Suppose that $d$ is an odd integer.
		  \begin{compactenum}

	\item[(c)] If $d\ge 5$, then $\mathrm{Sp}_{d+1}(R,I)\cap \ESp_{d+3}(R,I)= \ESp_{d+1}(R,I)$.
    \item[(d)] If $d\ge 7$, and $k$ is an algebraically closed field such that $\frac{1}{d!}\in k$, then $\mathrm{Sp}_{d-1}(R,I)\cap \ESp_{d+1}(R,I)= \ESp_{d-1}(R,I)$.
          \end{compactenum}

	\end{compactenum}
	\et
	
	\proof First, we assume that $d\geq 6$ is an even integer, and let $R$ be as in 1(a) or 1(b). In each case it is enough to show that hypothesis (\ref{H1}) of Proposition \ref{cis} is satisfied. Moreover, by Remark \ref{pciac}, it is enough to prove that for all $g\in \ld X^2-X\rd I[X]$ we must have \begin{equation}\label{H-con}
		\Um_{d+2}(R[X],\ld g \rd)=e_1\Sp_{d+2}(R[X], \ld g \rd).
	\end{equation}However, this follows from Theorem \ref{symplectic completion} (1) and Remark \ref{symplectic completionRmk}.
	
%
%
%

%
    
 Now, we assume that $d$ is an odd integer. If $d\ge 5$ such that $\rm{char}(k)\not=2$, then hypothesis~(\ref{H1}) follows from Corollary \ref{transitive action for ESp-alternating}. Therefore, we assume that $k$ is an algebraically closed field such that $\frac{1}{d!}\in k$. As in the previous case, the problem reduces to proving that for all $g\in \ld X^2-X\rd I[X]$ we must have \begin{equation}\label{H-con''}
 	\Um_{d+1}(R[X],\ld g \rd)=e_1\Sp_{d+1}(R[X], \ld g \rd).
 \end{equation}
However, this follows from Theorem \ref{symplectic completion} (2) and Remark \ref{symplectic completionRmk}. This concludes the proof.\qed


    In the remaining part of this section, we further improve the stability for the relative symplectic group $\mathrm{K_1Sp}(R[T,T^{-1}],I[T,T^{-1}])$. One may observe that when $k$ and $R$ are as described in Theorem \ref{Relative Stability for affine 2} 1(a) or (b), the result is contained there. However, for an arbitrary infinite base field, Suslin’s cancellation result~\cite[Theorem 2.4]{Suslin1984} is unavailable. Hence, in this setting, one cannot follow the philosophy of Rao–van der Kallen~\cite{rvdk}. Instead, to tackle this problem we rely on the algebraic structure of the Laurent polynomial ring and certain patching techniques. Before proceeding, we prove the following lemma, which we believe must be well known, although we could not find a suitable reference. For the sake of completeness, we give the proof.
	
	\bl
		\label{symplectic patching}
		Let $R$ be a ring and $I\subset R$ be an ideal. Let $s,t\in R\setminus I$ be such that $\langle s\rangle R+\langle t\rangle R=R$. Let $\alpha_1\in \Sp_{2n}(R_s, I_s)$ and $\alpha_2\in \Sp_{2n}(R_t, I_t)$ be such that $(\alpha_1)_t=(\alpha_2)_s$. Then there exists a unique $\alpha\in \Sp_{2n}(R,I)$ such that $\alpha_s=\alpha_1$ and $\alpha_t=\alpha_2$.	 				
		
	\el
	
	\proof By the universal property of the fiber product, there exists an $\alpha \in \GL_{2n}(R,I)$ such that $\alpha_s=\alpha_1$ and $\alpha_t=\alpha_2$. We claim that $\alpha \in \Sp_{2n}(R,I)$. For this, it suffices to verify that $\alpha^T \chi_n \alpha = \chi_n$, and it is further enough to check this locally. Since for any prime ideal $\mathfrak{p}\in \Spec(R)$ either $s \notin \mathfrak{p}$ or $t \notin \mathfrak{p}$, we have $\alpha_{\mathfrak{p}}\in \Sp_{2n}(R_{\mathfrak{p}}, I_{\mathfrak{p}})$. Hence $(\alpha^T \chi_n \alpha)_{\mathfrak{p}} = (\chi_n)_{\mathfrak{p}}$ for all $\mathfrak{p}\in \Spec(R)$. This concludes the proof.\qed

	\bt
		\label{main theorem}
		Let $d\geq 5$ be an odd integer. Let $R$ be a regular affine algebra of dimension $d$ over an infinite field $k$ such that $\rm{char}(k)\not=2$. Let $I$ be an ideal of $R$ such that $R/I$ is regular. Let $A=R[T,T^{-1}]$ and $J=I[T,T^{-1}]$. Then $$\Sp_{d+1}(A,J)\cap \ESp_{d+3}(A,J)=\ESp_{d+1}(A,J).$$
	\et
	
	\proof
Let us choose $\alpha \in \Sp_{d+1}(A,J) \cap \ESp_{d+3}(A,J)$. Since  $\ESp_{d+1}(A,J)\subset  \Sp_{d+1}(A,J) \cap \ESp_{d+3}(A,J) $, it therefore suffices to prove that $\alpha \in \ESp_{d+1}(A,J)$.
	 
	 We define $S:=k[T]\setminus\{0\}, B:=S^{-1}A$, and  $L:=S^{-1}J$. Then $B$ is a regular affine algebra of dimension $d$ over the field $k(T)$. Therefore, applying Theorem \ref{Relative Stability for affine 2} 2(c) we obtain that $\alpha_S \in \ESp_{d+1}(B,L)$. Clearing denominators one can find an $h\in k[T]\setminus\{0\}$ such that $\alpha_h\in \ESp_{d+1}(A_h,J_h)$. Since $T$ is a unit in $R[T,T^{-1}]$, without loss of generality, we may assume that $h(0)\neq 0$. Further, without loss of generality we may assume that $h$ is a monic polynomial in $T$.
	
	Since $h(0)\in k\setminus \{0\}$, we get that $\langle h\rangle k[T]+\langle T\rangle k[T]=k[T]$. Hence we have the following fiber product diagram.
	\[
	\begin{tikzcd}
		R[T]\arrow{d}\arrow{r}&R[T,T^{-1}]\arrow{d}\\
		R[T]_h\arrow{r} &R[T,T^{-1}]_h
	\end{tikzcd}
	\]
 As $\alpha_h\in \ESp_{d+1}(A_h,J_h)$, applying Corollary \ref{patching} or \cite[Lemma 3.7]{kopeiko} (when $I=R$) we can find $\varepsilon_1\in \ESp_{d+1}(A,J)$ and $\varepsilon_2\in \ESp_{d+1}(R[T]_h,I[T]_h)$ such that $(\alpha)_{h}=(\varepsilon_2)_T(\varepsilon_1)_h$. Let us define $$\alpha_1:=\alpha\varepsilon_1^{-1} \in \Sp_{d+1}(A,J)\text{ and }\alpha_2:= \varepsilon_2\in \ESp_{d+1}(R[T]_h, I[T]_h).$$ Then it follows from the construction  that $(\alpha_1)_h=(\alpha_2)_T$. Using Lemma \ref{symplectic patching}, we can find a unique element $\varepsilon\in \Sp_{d+1}(R[T], I[T])$  such that  $$(\varepsilon)_T=\alpha_1\text{ and }(\varepsilon)_h=\alpha_2.$$ However, we have $$(\varepsilon)_{Th}=(\alpha_2)_T\in \ESp_{d+1}(R[T]_{Th}, I[T]_{Th}),$$ where $Th$ is a monic polynomial in $R[T]$. Hence applying Corollary \ref{relative monic inversiion for symplectic matrix}, we obtain that $$\varepsilon\in \ESp_{d+1}(R[T], I[T]).$$Recall that, we have $$\alpha_1=\alpha\varepsilon_1^{-1} =(\varepsilon)_T\in \ESp_{d+1}(A,J)\text{ and }\varepsilon_1\in \ESp_{d+1}(A,J).$$ In particular, this proves that $\alpha\in \ESp_{d+1}(A,J)$. This concludes the proof.	\qed		

\section{Application: stably free modules over smooth real $4$-folds}\label{application}
In this section, we give an application of the theory developed in the preceding sections. Recall that the group $\widetilde{V_{SL}}(R)$ is the cokernel of the hyperbolic map
$$
\SK(R) \longrightarrow \widetilde{V}(R),
$$
where $\widetilde{V}(R)$ is a generalized Vaserstein symbol as defined in \cite{Tariq2019} (see also \cite[\S 3A]{Syed21}). The main theorem of this section is the following.

\bt\label{Tariq's theorem}
Suppose that $X=\Spec(R)$ is a smooth real affine variety of dimension $4$ such that $X(\mathbb R)=\emptyset$. Then stably free $R$-modules are free if and only if $\widetilde{V_{SL}}(R) = 0$.
\et

When the base field is algebraically closed, the above theorem was proved in \cite[Theorem 3.19]{Syed21}. To proceed, we need some preparation. The next result can be deduced from \cite{bfl} using the cohomological methods used in that paper (see Remark 4.12 there). However, the following argument provides a simpler alternative.

\bl\label{injective stability}
Suppose that $X=\Spec(R)$ is a smooth real affine variety of dimension $d\ge 3$ such that $X(\mathbb R)=\emptyset$. Then the canonical map $\frac{\SL_d(R)}{\E_d(R)} \to \SK(R)$ is injective.
\el

\proof In view of Proposition \ref{cis}, it is enough to show the following.

$$\Um_{d+1}(R[X],\ld X^2 - X\rd ) = e_1\SL_{d +1}( R[X],\ld X^2 - X\rd )$$
The proof is devoted to establishing this only. We follow the proof given in \cite[Corollary 7.7]{FRS}. Let $v(X)\in \Um_{d+1}(R[X],\ld X^2 - X\rd )$. We define $A := R[X,T]/\ld T^2 - Tf \rd$, where $f=X^2-X$. Since $v(0)=v(1)=e_1$, there exists $w(X)\in A[X]^{d+1}$ such that $v(X)=e_1+fw(X)$. We define $u(T):= e_1+Tw(X)$, then note that $u(T)$ is in $\Um_{d+1}(A)$ with $u(0)=e_1$ and $u(f)=v(X)$. We observe that $R[X] \hookrightarrow A$ is an integral extension. Therefore, the algebra $A$ also does not have real maximal ideals. Furthermore, as $R$ is smooth, the ideal defining the singular locus of $A$ is precisely $\ld T,f \rd$. Since $u(0)=e_1$, it follows that $u(T)$ lies in the same elementary orbit of $e_1$ modulo $\ld T,f\rd$. Hence by \cite[Theorem 4.13]{bfl} (also see Remark 4.14 in the same article) one can find an $\alpha(T)\in \SL_{d+1}(A)$ such that $u(T)=e_1\alpha(T)$. Let $\beta(T):=\alpha(0)^{-1}\alpha(T)$. Then $\beta(f) \in \SL_{d+1}(R[X],\ld X^2-X\rd )$ such that $v(X)=e_1\beta(f)$. This concludes the proof.\qed

\bp\label{SymplecticOrbitSpace}
Suppose that $X=\Spec(R)$ is a smooth real affine variety of even dimension $d\ge 4$ such that $X(\mathbb R)=\emptyset$. Let $\gamma\in \GL_{d}(R)$ be a non-degenerate alternating form. Then the orbit space $\frac{\Um_d(R)}{\Sp_{\gamma}(R)}$ is trivial.
\ep
\proof First, we note that, since \({\Sp}_{\gamma}(R)={\Sp}_{u\gamma}(R)\) for every \(u\in A^\times\), without loss of generality, we may assume that $\operatorname{Pf}(\gamma)=1$. Let $v\in \Um_d(R)$. It is enough to show that all hypotheses of Proposition \ref{sufficient condition for alternating completion} are satisfied. We first note that hypothesis (1) is satisfied from Lemma \ref{injective stability}, and hypothesis (3) is satisfied as $$\Um_{d+2i}(R)=e_1\E_{d+2i}(R)=e_1\ESp_{\chi_i\perp \gamma}(R)$$ for all $i\ge 1$ (cf. Lemma \ref{relative version of Lemma 5.5 of SV}). Therefore, we only need to find an $\alpha\in \SL_d(R)$ such that (a) $v = e_1\alpha$ and (b) $\alpha \in \operatorname{im}(f)$, where $f :
\KSp(R) \to \SK(R)$ is the forgetful map. In the following we only show this.

It follows from the proof of \cite[Theorem 4.13]{bfl} that there exists a unimodular row $(a_1,\ldots, a_d)$ in $\Um_d(R)$ such that $v\equiv (a_1,\ldots, a_{d-1},a_d^{(d!)^2})\mod \E_d(R)$, and hence using Lemma \ref{relative version of Lemma 5.5 of SV} we get that $$v\equiv (a_1,\ldots, a_{d-1},a_d^{(d!)^2})\mod \ESp_d(R).$$
Let $v= (a_1,\ldots, a_{d-1},a_d^{(d!)^2})\alpha_1$ for some $\alpha_1\in \ESp_d(R)$. Then using Lemma \ref{completion of coordinate power lies in im(f)} we can find an $\alpha_2 \in \SL_d(R)$ such that $(a_1,\ldots, a_{d-1},a_d^{(d!)^2}) = e_1\alpha_2$, and $[\alpha_2] \in \SK(R)$ lies in the image of $f$. Let us take $\alpha=\alpha_2\alpha_1\in \SL_d(R)$. Then $v=e_1\alpha $, and $\alpha\in \operatorname{im}(f)$. This concludes the proof.\qed

\textbf{Proof of Theorem \ref{Tariq's theorem}.} First, we note that any stably free \(R\)-module of rank greater than \(2\) is already free by \cite[Theorem 4.13]{bfl}, while stably free \(R\)-modules of rank one are also free. Therefore, it remains only to consider stably free \(R\)-modules of rank \(2\). By \cite[Theorem 4.13]{bfl}, any stably free $R$-module of rank $2$ is $1$-stably free, hence it is free if and only if the orbit space $\frac{\Um_3(R)}{\SL_3(R)}$ is trivial. Now the hypotheses of \cite[Theorem 3.2 and 3.6]{Syed21} are follow from  \cite[Theorem 2.8]{Banerjee2025} and Proposition \ref{SymplecticOrbitSpace}, respectively. Hence, applying the same, it follows that $\frac{\Um_3(R)}{\SL_3(R)}$ is trivial if and only if  $\widetilde{V_{SL}}(R) = 0$. This concludes the proof.\qed

\tiny
		\bibliographystyle{abbrvurl}
	\bibliography{Symplectic_stability}

	\end{document}